\documentclass[11pt]{article}
\usepackage{amsmath, amscd, amssymb, latexsym, epsfig, color, amsthm}
\numberwithin{equation}{section}

\def\endproof{\hfill$\square$\medskip}
\newtheorem{theorem}{Theorem}[section]
\newtheorem{proposition}[theorem]{Proposition}

\newtheorem{corollary}[theorem]{Corollary}

\newtheorem{lemma}[theorem]{Lemma}

\theoremstyle{definition}
\newtheorem{definition}[theorem]{Definition}

\newtheorem{remark}[theorem]{Remark}
\newtheorem{notation}[theorem]{Notation}

\DeclareMathOperator\lk{\mathrm{lk}}
\DeclareMathOperator\st{\mathrm{st}}

\newcommand{\field}{{\mathbb F}}

\newcommand{\Q}{{\mathbb Q}}
\newcommand{\Z}{{\mathbb Z}}
\newcommand{\ZZ}{{\mathbb Z}}
\newcommand{\FF}{{\mathbb F}}

\newcommand{\Sm}{{\mathcal S}}

\newcommand{\mideal}{\ensuremath{\mathfrak{m}}}
\newcommand{\Tor}{\ensuremath{\mathrm{Tor}}\hspace{1pt}}

\newcommand{\Ker}{\ensuremath{\mathrm{Ker}}\hspace{1pt}}
\newcommand{\Image}{\ensuremath{\mathrm{Im}}\hspace{1pt}}

\title{Face numbers of manifolds with boundary}
\author{Satoshi Murai\thanks{Research is partially
supported by JSPS KAKENHI 25400043.}\\
\small Department of Pure and Applied Mathematics\\[-0.8ex]
\small Graduate School of Information Science and Technology\\[-0.8ex]
\small Osaka University, Suita, Osaka 565-0871, Japan\\[-0.8ex]
\small \texttt{s-murai@ist.osaka-u.ac.jp}
\and Isabella Novik\thanks{Research is partially
supported by NSF grant DMS-1361423}\\
\small Department of Mathematics\\[-0.8ex]
\small University of Washington\\[-0.8ex]
\small Seattle, WA 98195-4350, USA\\[-0.8ex]
\small \texttt{novik@math.washington.edu}
}

\begin{document}
\maketitle

\begin{abstract}
We study face numbers of simplicial complexes that triangulate manifolds (or even normal pseudomanifolds) with boundary. Specifically, we establish a sharp lower bound on the number of interior edges of a simplicial normal pseudomanifold with boundary in terms of the number of interior vertices and relative Betti numbers. Moreover, for triangulations of manifolds with boundary all of whose vertex links have the weak Lefschetz property, we extend this result to sharp lower bounds on the number of higher-dimensional interior faces. Along the way we develop a version of Bagchi and Datta's $\sigma$- and $\mu$-numbers for the case of relative simplicial complexes and prove stronger versions of the above statements with the Betti numbers replaced by the $\mu$-numbers. Our results provide natural generalizations of known theorems and conjectures for closed manifolds and appear to be new even for the case of a ball. 
\end{abstract}

\noindent{{\large{\bf Keywords:}} face numbers, the lower bound theorem,
triangulations of manifolds, relative
 simplicial complexes, Stanley-Reisner modules,
graded Betti numbers, the weak
Lefschetz property, Morse inequalities.

\section{Introduction}
Given a simplicial complex $\Delta$, one can count the number of faces of $\Delta$ of each dimension. These numbers are called the face numbers or the $f$-numbers of $\Delta$. When $\Delta$ triangulates a manifold or a normal pseudomanifold $M$, it is natural to ask what restrictions does the topology of $M$ place on the possible face numbers of $\Delta$. For the case of closed manifolds, the last decade of research led to tremendous progress on this question, see, for instance, \cite{Bagchi:mu-vector, Bagchi-Datta-14, Lutz-Sulanke-Swartz-09, Murai-10:Barycentric, Murai-15, Murai-Nevo-14, Novik-05, Novik-Swartz-09:DS, Novik-Swartz-09:Gorenstein, Novik-Swartz-09:Socles, Swartz-09, Swartz-14}. On the other hand, face numbers of manifolds with boundary remained a big mystery, and very few papers even touched on this subject, see \cite{Grabe-87, Kolins-11, Murai-Nevo-14, Novik-Swartz-09:DS}. At present there is not even a conjecture for characterizing the set of $f$-vectors of balls of dimension six and above, see \cite{Kolins-11}.

The goal of this paper is to at least partially remedy this situation. The main simple but surprisingly novel idea the paper is built on is that to study the face numbers of a manifold with boundary $\Delta$, the right object to analyze is the \emph{relative simplicial complex} $(\Delta, \partial\Delta)$ rather than the complex $\Delta$ itself. (We must mention that earlier this year an idea of using relative simplicial complexes was applied by Adiprasito and Sanyal in their breakthrough solution of long-standing questions regarding the combinatorial complexity of Minkowski sums of polytopes, see \cite{Adiprasito-Sanyal}. Our paper is undoubtedly influenced by their results.)

Once this simple realization is made, the rest of the pieces fall, with some work, into place. For instance, the Dehn-Sommerville relations from \cite{Grabe-87} take on the following elegant form. (We defer all the definitions until the next section, and for now merely note that the $h''$-numbers are linear combinations of the $f$-numbers and the Betti numbers.)
\begin{proposition} \label{Dehn-Somm}
Let $\Delta$ be a (not necessarily connected) $(d-1)$-dimensional orientable homology manifold with boundary. Then 
\[
h''_i(\Delta,\partial\Delta)=h''_{d-i}(\Delta) \quad \mbox{for all } 0<i < d.
\]
If $\Delta$ is connected, then this equality also holds for $i=0$ and $i=d$.
\end{proposition}

\noindent When $\Delta$ is a homology $(d-1)$-ball, Proposition \ref{Dehn-Somm} reduces to the fact that $h_i(\Delta,\partial\Delta)=h_{d-i}(\Delta)$ for all $0\leq i\leq d$. This case of Proposition \ref{Dehn-Somm} is well-known, see \cite[II Section 7]{Stanley-96}.

To state our main results, we first recall the famous Lower Bound Theorem of Barnette \cite{Barnette-73}, Kalai \cite{Kalai-87}, Fogelsanger \cite{Fogelsanger-88}, and Tay \cite{Tay-95} asserting that if $\Delta$ is a connected normal pseudomanifold without boundary of dimension at least two, then $g_2(\Delta)\geq 0$; furthermore, if $\dim\Delta\geq 3$, then $g_2(\Delta)=0$ if and only if $\Delta$ is a stacked sphere. This theorem was recently significantly strengthened to bound $g_2(\Delta)$ by certain topological invariants of $\Delta$ as follows. Below we denote by $\tilde{b}_i(-; \field):=\dim_\field \tilde{H}_i(-;\field)$ the $i$-th reduced Betti number computed over a field $\field$.

\begin{theorem} \label{LBT-closed} {\rm \cite[Theorem 5.3]{Murai-15}} 
Let $\Delta$ be a (not necessarily connected) $d$-dimensional normal pseudomanifold without boundary. If $d\geq 3$ then
\begin{equation} \label{rig-ineq1}
g_2(\Delta)\geq \binom{d+2}{2} \left(\tilde{b}_1(\Delta; \field)-\tilde{b}_0(\Delta; \field)\right).
\end{equation}
Moreover, $g_2(\Delta)= \binom{d+2}{2} \left(\tilde{b}_1(\Delta; \field)-\tilde{b}_0(\Delta; \field)\right)$ if and only if each connected component of $\Delta$ is a stacked manifold.
\end{theorem}
For manifolds, this result was conjectured by Kalai \cite{Kalai-87}; it was proved in the above generality by the first author. For orientable homology manifolds, the inequality part was originally verified in \cite[Theorem 5.2]{Novik-Swartz-09:Socles}. 
(In \cite{Murai-15,Novik-Swartz-09:Socles}, it was assumed that $\Delta$ is connected, 
but the disconnected case follows easily from the connected one.)

How does the situation change for manifolds with boundary? Our first main result is that the inequality part applies almost verbatim to normal pseudomanifolds with boundary: the only adjustment we need to make is to replace $\Delta$ with the relative complex $(\Delta, \partial\Delta)$ and the Betti numbers of $\Delta$ with the relative Betti numbers of $(\Delta, \partial\Delta)$. (Note that if $\Delta$ has no boundary, then $(\Delta, \partial\Delta)=(\Delta,\emptyset)=\Delta$.) More precisely, we have:

\begin{theorem} \label{main1}
Let $\Delta$ be a (not necessarily connected) $d$-dimensional normal pseudomanifold with boundary. If $d\geq 3$ then
\begin{align} \label{main1-inequality}
g_2(\Delta,\partial \Delta)\geq \binom{d+2}{2} \left(\tilde{b}_1(\Delta, \partial\Delta; \field)-\tilde{b}_0(\Delta,\partial\Delta; \field)\right).
\end{align}
Furthermore, if $\Delta$ is an $\field$-homology manifold with boundary
for which \eqref{main1-inequality} is an equality
then the link of each interior vertex of $\Delta$ is a stacked sphere and the link of each boundary vertex of $\Delta$ is obtained from an $\field$-homology ball that has no interior vertices by forming connected sums with the boundary complexes of simplices.
\end{theorem}

As an easy consequence we obtain

\begin{corollary} \label{h_2-vs-int-vert}
Let $\Delta$ be a (not necessarily connected) $d$-dimensional normal pseudomanifold with boundary whose boundary is also a normal pseudomanifold.
If $d \geq 4$, then
\begin{align*}
h_2(\Delta)\geq f_0(\Delta,\partial\Delta)\!+\!\binom{d+2}{2}\! \left(\tilde{b}_1(\Delta, \partial\Delta; \field)\!-\!\tilde{b}_0(\Delta,\partial\Delta; \field)\right) \!+\! 
\binom{d+1}{2}\! \left(\tilde{b}_1(\partial\Delta; \field)\!-\!\tilde{b}_0(\partial\Delta; \field)\right),
\end{align*}
and if $d=3$, then
$$h_2(\Delta) \geq f_0(\Delta,\partial \Delta)+
10\left(\tilde{b}_1(\Delta, \partial\Delta; \field)-\tilde{b}_0(\Delta, \partial\Delta; \field)\right) + 3\left(\tilde{b}_1(\partial\Delta; \field)-2\tilde{b}_0(\partial\Delta; \field)\right).$$
\end{corollary}

\noindent For homology manifolds with boundary this result strengthens \cite[Theorem 5.1]{Novik-Swartz-09:DS}, which in itself is a strengthening of \cite[Theorem 11.1]{Kalai-87}.

As for higher-dimensional face numbers, it was proved in \cite[Eq.~(9)]{Novik-Swartz-09:DS} that if $\Delta$ is a $d$-dimensional \emph{orientable} $\field$-homology manifold (without boundary) all of whose vertex links have the weak Lefschetz property (WLP, for short), then the following generalization of eq.~\eqref{rig-ineq1} holds:
\[
g_r(\Delta)\geq \binom{d+2}{r}\sum_{j=1}^r (-1)^{r-j}\cdot \tilde{b}_{j-1}(\Delta; \field) \quad \mbox{for all } 1\leq r\leq (d+1)/2.
\]
Our second main result asserts that the relative version of the same statement applies to \textbf{all} (orientable or non-orientable) homology manifolds with or without boundary all of whose vertex links have the WLP:

\begin{theorem} \label{main2}
Let $\Delta$ be a $d$-dimensional $\field$-homology manifold with or without boundary. If all vertex links of $\Delta$ have the WLP, then 
\[
g_r(\Delta, \partial\Delta)\geq \binom{d+2}{r}\sum_{j=1}^r (-1)^{r-j}\cdot\tilde{b}_{j-1}(\Delta, \partial\Delta; \field) \quad \mbox{for all } 1\leq r\leq (d+1)/2.
\]
Moreover, if equality holds for some $r \leq d/2$, then the link of each interior vertex of $\Delta$ is an $(r-1)$-stacked homology sphere.
\end{theorem}

Theorems \ref{main1} and \ref{main2} are sharp, see Remark \ref{sharpness}.
In the case of orientable manifolds with boundary, the proofs of Theorems \ref{main1} and \ref{main2} follow the same ideas as the proofs of \cite[Eq.~(9)]{Novik-Swartz-09:DS} and \cite[Theorem 5.2]{Novik-Swartz-09:Socles}. The treatment of the non-orientable case requires much more work: it requires (i) developing a version of Bagchi and Datta's $\sigma$- and $\mu$-numbers introduced in \cite{Bagchi-Datta-14} (see also \cite{Bagchi:mu-vector}) for relative simplicial complexes, and (ii) refining the methods from \cite{Murai-15} to bound certain alternating sums of the graded Betti numbers of the Stanley-Reisner modules. In fact, we prove stronger versions of Theorems \ref{main1} and \ref{main2} with the reduced Betti numbers replaced by the $\mu$-numbers, see Theorems \ref{5.5} and \ref{6.3} (along with Corollary \ref{mu-equal-charact}).
These results, especially Theorem \ref{5.5}, appear to be new even for the case of balls and spheres, see Remark \ref{remark-new}. Since we are using two different methods in our proofs, it is sometimes more convenient for us to work with $(d-1)$-dimensional complexes and other times with $d$-dimensional ones. To avoid any confusion, the dimension used is explicitly stated in each result.

The structure of the rest of the paper is as follows. In Section 2 we review basics of simplicial complexes and relative simplicial complexes, and also of Stanley-Reisner rings and modules. Section 3 serves as a warm-up for the rest of the paper. There we verify Proposition \ref{Dehn-Somm}, prove the inequality parts of Theorems \ref{main1} and \ref{main2} for the case of orientable homology manifolds, and derive Corollary \ref{h_2-vs-int-vert} from Theorem \ref{main1}. In Section 4 we develop an analogue of $\sigma$- and $\mu$-numbers for relative simplicial complexes. In Section 5 we derive upper bounds on certain (alternating sums of) graded Betti numbers of Stanley-Reisner modules of normal pseudomanifolds with boundary and also of Stanley-Reisner rings and modules of homology spheres and homology balls with the WLP. Using results of Sections 4 and 5, we prove (a strengthening of) Theorem~\ref{main2} in Section 6 and (a strengthening of) Theorem~\ref{main1} in Section 7.  In Section 8, we establish a criterion characterizing homology $d$-balls with $g_2(\Delta, \partial\Delta)=0$.
We close in Section 9 by showing that, in contrast with the Betti numbers, the $\mu$-numbers always detect the non-vanishing of the fundamental group; we also provide several additional remarks and open problems.


\section{Preliminaries} \label{sect:preliminaries}
In this section we review several basic definitions and results on simplicial complexes and relative simplicial complexes, as well as on the Stanley-Reisner rings and modules.

A \textbf{simplicial complex} $\Delta$ on a (finite) ground set $V = V(\Delta)$ is a collection of subsets of $V(\Delta)$ that is closed under inclusion. (We do not assume that every singleton $\{v\}\subseteq V$ is an element of $\Delta$.) The elements $F \in \Delta$ are called \textbf{faces} and the maximal faces of $\Delta$ under inclusion are called \textbf{facets}. We say that $\Delta$ is \textbf{pure} if all of its facets have the same cardinality. The \textbf{dimension} of a face $F \in \Delta$ is $\dim(F) = |F|-1$ and the dimension of $\Delta$ is $\dim(\Delta) = \max\{\dim(F)\ : \ F \in \Delta\}$. We refer to $i$-dimensional faces as \textbf{$i$-faces}; the $0$-faces are also called \textbf{vertices}. 

If $\Delta$ is a simplicial complex and $F \in \Delta$ is a face, the local structure of $\Delta$ around $F$ is described by the (closed) \textbf{star} of $F$ in $\Delta$: $\st_{\Delta}(F):= \{G \in \Delta\ : \ F \cup G \in \Delta\}$. Similarly, the \textbf{link} of $F$ in $\Delta$ and the \textbf{deletion} of $F$ from $\Delta$ are defined as
$$\lk_{\Delta}(F):= \{G \in \st_\Delta(F)\ : F \cap G = \emptyset\} \quad \mbox{and} \quad \Delta\setminus F:=\{G \in \Delta\ : \ F\not\subseteq G\}.$$
(Our convention is that $\lk_\Delta(F)=\emptyset$ if $F \not \in \Delta$.)
When $F = \{v\}$ is a single vertex, we write $\lk_{\Delta}(v)$ and $\Delta\setminus v$ in place of $\lk_{\Delta}(\{v\})$ and $\Delta\setminus\{v\}$. 

Let $\Delta'$ and $\Delta''$ be pure simplicial complexes of the same dimension on disjoint ground sets. Let $F'$ and $F''$ be facets of $\Delta'$ and $\Delta''$ respectively, and let $\varphi: F' \rightarrow F''$ be a bijection between the vertices of $F'$ and the vertices of $F''$. The \textbf{connected sum} of $\Delta'$ and $\Delta''$, denoted $\Delta' \#_{\varphi} \Delta''$ or simply $\Delta' \# \Delta''$, is the simplicial complex obtained by identifying the vertices of $F'$ and $F''$ (and all faces on those vertices) according to the bijection $\varphi$ and removing the facet corresponding to $F'$ (which has been identified with $F''$). 

Given a pair of simplicial complexes $\Gamma\subseteq \Delta$, we let $\Psi:=(\Delta,\Gamma)$ be the corresponding \textbf{relative simplicial complex}: the faces of $\Psi$ are precisely the faces of $\Delta$ not contained in $\Gamma$ and the dimension of $\Psi$ is the maximum dimension of its faces. For instance, $(\Delta,\emptyset)=\Delta$.

Let $\field$ be a field. For a simplicial complex $\Delta$ of dimension $d-1$, let
\begin{align*}C_\bullet(\Delta;\field) &: \; 0 \longrightarrow C_{d-1}(\Delta) \longrightarrow \cdots \longrightarrow C_1(\Delta) \longrightarrow C_0(\Delta) \longrightarrow 0 \qquad \mbox{and} \\
\tilde{C}_\bullet(\Delta;\field)&: \; 0 \longrightarrow C_{d-1}(\Delta) \longrightarrow \cdots \longrightarrow C_1(\Delta) \longrightarrow C_0(\Delta) \longrightarrow C_{-1}(\Delta) \longrightarrow 0
\end{align*}
be the simplicial chain complex and the reduced simplicial chain complex of $\Delta$ with coefficients in $\field$. Here $C_k(\Delta)$ is the vector space over $\field$ with basis $\{e_G: G \in \Delta,\ |G|=k+1\}$.
In particular, $C_{-1}(\Delta)$ is $1$-dimensional if $\Delta\neq \emptyset$ and it is $0$-dimensional if $\Delta=\emptyset$.
For a relative simplicial complex $\Psi=(\Delta,\Gamma)$ and $k\geq -1$, we define
$C_k(\Psi;\field)=C_k(\Delta;\field)/C_k(\Gamma)$; as above, it gives rise to simplicial and reduced simplicial chain complexes $C_\bullet(\Psi;\field)$ and $\tilde C_\bullet(\Psi;\field)$. We denote by ${H}_i(-; \field)=H_i(C_\bullet(-;\field))$ and $\tilde{H}_i(-; \field)=H_i(\tilde C_\bullet(-;\field))$, respectively, the $i$-th homology and the $i$-th reduced homology computed with coefficients in $\field$, and by $b_i(-;\field)$ and $\tilde{b}_i(-;\field)$, respectively, the dimensions of ${H}_i(-; \field)$ and $\tilde{H}_i(-; \field)$ over $\field$. When $\field$ is fixed, we often omit it from our notation. Note that for all $i>0$, $\tilde{b}_i(\Delta)=b_i(\Delta)$ and $\tilde{b}_i(\Delta,\Gamma)={b}_i(\Delta,\Gamma)$; on the other hand, $\tilde{b}_{-1}(\{\emptyset\})=1$ while $\tilde{b}_{-1}(\Delta)=0$ if $\Delta\neq\{\emptyset\}$; similarly, $\tilde{b}_0(\Delta,\emptyset)=\tilde{b}_0(\Delta)=b_0(\Delta)-1=b_0(\Delta,\emptyset)-1$ if $\dim\Delta\geq 0$, while $\tilde{b}_0(\Delta,\Gamma)={b}_0(\Delta,\Gamma)$ if $\Gamma\neq \emptyset$.

The complexes we study in this paper are homology manifolds and normal pseudomanifolds with or without boundary. Below we denote by $\mathbb{S}^j$ and $\mathbb{B}^j$ the $j$-dimensional sphere and ball, respectively. A $(d-1)$-dimensional simplicial complex $\Delta$ is an \textbf{$\field$-homology sphere} if $\tilde{H}_*(\lk_{\Delta}(G);\field) \cong \tilde{H}_*(\mathbb{S}^{d-|G|-1};\field)$ for all faces $G \in \Delta$ (including $G = \emptyset$). Similarly, $\Delta$ is an \textbf{$\field$-homology manifold} (without boundary) if for all nonempty faces $G\in\Delta$, the link of $G$ is an $\field$-homology $(d-|G|-1)$-sphere; in this case we write $\partial\Delta=\emptyset$. Homology manifolds without boundary are sometimes referred to as \emph{closed homology manifolds}.

Homology manifolds with boundary are defined in an analogous way: a $d$-dimensional simplicial complex $\Delta$ is an \textbf{$\field$-homology manifold with boundary} if (1) the link of each nonempty face $G$ of $\Delta$ has the homology of either $\mathbb{S}^{d-|G|}$ or $\mathbb{B}^{d-|G|}$, and (2) the set of all \textbf{boundary faces}, that is,
\[
\partial\Delta:= \left\{G\in\Delta \ : \ \tilde{H}_*(\lk_{\Delta}(G);\field) \cong \tilde{H}_*(\mathbb{B}^{d-|G|};\field)\right\} \cup \{\emptyset\},
\]
is a $(d-1)$-dimensional $\field$-homology manifold without boundary. 
We also mention that $\Delta$ is an $\field$-homology $d$-ball if (1) $\Delta$ is a homology manifold with boundary, (2) $\tilde H_* (\Delta;\field) \cong \tilde H_*(\mathbb B^d;\field)$, and (3) the boundary of $\Delta$ is a homology sphere. For instance, the link of any boundary vertex of a homology manifold with boundary is a homology ball.

A connected $d$-dimensional $\field$-homology manifold (with or without boundary) is called \textbf{orientable} if $\tilde{H}_d(\Delta,\partial\Delta;\field)\cong \field$. A disconnected $\field$-homology manifold is orientable if each of its connected components is.

A $d$-dimensional simplicial complex $\Delta$ is a \textbf{normal pseudomanifold} (with or without boundary) if (1) it is pure, (2) each $(d-1)$-face (or \textbf{ridge}) of $\Delta$ is contained in at most two facets of $\Delta$, and (3) the link of each nonempty face of dimension at most $d-2$ is connected. Such a complex $\Delta$ is called a \textbf{normal pseudomanifold without boundary} if each ridge of $\Delta$ is contained in exactly two facets of $\Delta$, and it is called a \textbf{normal pseudomanifold with boundary} if there is a ridge contained in only one facet of $\Delta$. If $\Delta$ is a normal pseudomanifold with boundary, then the boundary of $\Delta$, $\partial\Delta$, is defined to be the pure $(d-1)$-dimensional complex whose facets are precisely the ridges of $\Delta$ that are contained in unique facets of $\Delta$. We note that every homology manifold (with or without boundary) is a normal pseudomanifold (with or without boundary). 

Let $\Psi=(\Delta, \Gamma)$ be a $(d-1)$-dimensional relative simplicial complex. The main object of our study is the \textbf{$f$-vector} of $\Psi$, $f(\Psi):=(f_{-1}(\Psi), f_0(\Psi),\ldots, f_{d-1}(\Psi))$, where $f_i=f_i(\Psi)$ denotes the number of $i$-faces of $\Psi$; the numbers $f_i$ are called the \textbf{$f$-numbers} of $\Psi$. Observe that if $\Gamma=\emptyset$ and $\Delta\neq\emptyset$, then $f_{-1}(\Psi)=1$, while if $\Gamma\neq\emptyset$, then $f_{-1}(\Psi)=0$. Also, if $\Delta$ is a normal pseudomanifold with boundary and $\Gamma=\partial\Delta$, then $f_i(\Psi)$ counts the number of \textbf{interior} $i$-dimensional faces of $\Delta$.

For several algebraic reasons, it is often more natural to study a certain (invertible) integer transformation of the $f$-vector called the \textbf{$h$-vector} of $\Psi$, $h(\Psi)=(h_0(\Psi),h_1(\Psi),\ldots,h_d(\Psi))$: its components, the \textbf{$h$-numbers}, are defined by
\begin{equation} \label{h-numbers}
h_j(\Psi):= \sum_{i=0}^j (-1)^{j-i} \binom{d-i}{d-j} f_{i-1}(\Psi) \quad \mbox{for all} \quad 0 \leq j \leq d=\dim\Psi+1.
\end{equation}
We also define the \textbf{$g$-numbers} of $\Psi$ by $g_0(\Psi):=h_0(\Psi)$ and $g_j(\Psi):=h_j(\Psi)-h_{j-1}(\Psi)$ for $j>0$.

Assume that $\Delta$ is a $(d-1)$-dimensional simplicial complex with $|V(\Delta)|=n$ and $\field$ is an \emph{infinite} field of an arbitrary characteristic. We identify $V(\Delta)$ with $[n] = \{1,2,\ldots,n\}$. Consider a polynomial ring $S:=\field[x_1,x_2,\ldots,x_n]$ with one variable for each element of $V(\Delta)$. The \textbf{Stanley-Reisner ideal} of $\Delta$ is the ideal $$I_{\Delta}:= \left(x_{i_1}x_{i_2}\cdots x_{i_k} \ : \ \{i_1,i_2,\ldots,i_k\} \notin \Delta \right) \subseteq S.$$ The \textbf{Stanley-Reisner ring} (or face ring) of $\Delta$ is the quotient $\field[\Delta]:= S/I_{\Delta}$. Similarly, for a relative simplicial complex $\Psi=(\Delta,\Gamma)$, the \textbf{Stanley-Reisner module} of $\Psi$ is $$\field[\Delta, \Gamma]:=I_{\Gamma}/I_{\Delta}.$$ 

If $\Delta$ is $(d-1)$-dimensional, then the Krull dimension of $\field[\Delta]$ is $d$. A sequence of linear forms, $\theta_1, \ldots, \theta_d \in S$ is called a \textbf{linear system of parameters} (or l.s.o.p.)~of $\field[\Delta]$ if the ring $\field[\Delta]/\Theta\field[\Delta]$ is a finite-dimensional $\field$-vector space; here $\Theta=(\theta_1, \ldots, \theta_d)$. For instance, a sequence of $d$ \emph{generic} linear forms provides an l.s.o.p.

Since $I_{\Delta}$ and $I_{\Gamma}$ are monomial ideals, the quotient ring $\field[\Delta]$ and the quotient module $\field[\Delta, \Gamma]$ are graded by degree. The $i$-th graded piece of a graded ring (or module) $R$ is denoted by $R_i$. The following result is due to Schenzel \cite{Schenzel-81}:

\begin{theorem} 
\label{thm:Schenzel1}
Let $\Delta$ be an $\field$-homology manifold (with or without boundary) of dimension $d-1$, and let $\theta_1, \ldots, \theta_d$ be an l.s.o.p.~of $\field[\Delta]$. Then
\[
\dim_{\field} \left(\field[\Delta]/\Theta\field[\Delta]\right)_j = h_j(\Delta) + \binom{d}{j} \sum_{i=1}^{j-1} (-1)^{j-i-1}\cdot \tilde{b}_{i-1}(\Delta; \field) \quad \mbox{for all} \quad 0\leq j\leq d.
\]
\end{theorem}

In light of this result, we define the $h'$-numbers of a $(d-1)$-dimensional relative simplicial complex $\Psi=(\Delta, \Gamma)$ as 
\begin{equation} \label{h'}
h'_j(\Psi) := h_j(\Psi) + \binom{d}{ j} \sum_{i=1}^{j-1} (-1)^{j-i-1}\cdot \tilde{b}_{i-1}(\Psi; \field) \quad \mbox{ for }0\leq j\leq d.
\end{equation}
(Note that $h'_d(\Psi)=\tilde{b}_{d-1}(\Psi;\field)$.) Also, in view of results from \cite{Novik-Swartz-09:Socles}, we define the \textbf{$h''$-numbers} of $\Psi$ as
\begin{equation} \label{h''}
h''_j(\Psi) := 
\begin{cases}
h'_j(\Psi) - \binom{d}{j} \tilde{b}_{j-1}(\Psi;\field) & \text{ if } 0 \leq j < d, \\
h'_d(\Psi) & \text{ if } j = d.
\end{cases}
\end{equation}

Let $\Delta$ be a $(d-1)$-dimensional $\field$-homology ball or an $\field$-homology sphere. We say that $\Delta$ has the \textbf{weak Lefschetz property} (over $\field$), abbreviated the WLP, if for $d+1$ generic linear forms $\theta_1, \ldots, \theta_d, \omega \in S$ the multiplication map $$\times\omega: \left(\field[\Delta]/\Theta\field[\Delta]\right)_{\lfloor d/2\rfloor} \to \left(\field[\Delta]/\Theta\field[\Delta]\right)_{\lfloor d/2\rfloor +1}$$ is onto. (Equivalently, $\times\omega: \left(\field[\Delta]/\Theta\field[\Delta]\right)_{j} \to \left(\field[\Delta]/\Theta\field[\Delta]\right)_{j+1}$ is onto for all $j\geq \lfloor d/2\rfloor$.) The boundary complex of any simplicial $d$-polytope has the WLP over $\Q$ \cite{Stanley-80}; furthermore, any $(d-1)$-dimensional ball that is contained in the boundary of a simplicial $d$-polytope has the WLP over $\Q$ \cite{Stanley-93}.

A $d$-dimensional $\FF$-homology manifold with boundary is said to be \textbf{$r$-stacked} if it has no interior faces of dimension $\leq d-r-1$; a $(d-1)$-dimensional $\FF$-homology manifold without boundary ($\FF$-homology sphere, respectively) is \textbf{$r$-stacked} if it is the boundary of an $r$-stacked homology $d$-manifold ($\FF$-homology $d$-ball, respectively).
The $1$-stacked $\FF$-homology manifolds are usually simply called \textbf{stacked homology manifolds}. It is well-known and easy to see that an $\FF$-homology sphere is stacked if and only if it is the connected sum of the boundary complexes of several simplices. In particular, stacked homology spheres are combinatorial spheres.

For homology spheres with the WLP, the following characterization of $r$-stackedness was established in \cite{Murai-Nevo-13}.

\begin{theorem}\label{GLBTsphere}
Let $\Delta$ be a $(d-1)$-dimensional $\FF$-homology sphere with the WLP.
Then for $1 \leq r \leq d/2$, $\Delta$ is $(r-1)$-stacked if and only if $g_{r}(\Delta)=0$.
\end{theorem}


\section{A warm up: orientable homology manifolds with boundary}
In this section we prove Proposition \ref{Dehn-Somm} and establish the inequality parts of Theorems \ref{main1} and \ref{main2} for the case of orientable homology manifolds. First we require the following easy observation.

\begin{lemma} \label{h-and-g}
Let $\Delta$ be a normal pseudomanifold with nonempty boundary $\partial \Delta$.
Then $$h_i(\Delta)=h_i(\Delta, \partial \Delta)+g_i(\partial \Delta) \quad \mbox{for all } i\geq 0.$$
\end{lemma}
\begin{proof} Note that $f_i(\Delta)=f_i(\Delta,\partial \Delta)+f_i(\partial \Delta)$ for all $i\geq -1$, and that $\dim\Delta=\dim(\Delta,\partial\Delta)=1+ \dim(\partial\Delta)$. The result now follows easily from the definition of the $h$- and $g$-numbers, see eq.~\eqref{h-numbers}.
\end{proof}

Recall also that the \textbf{reduced Euler characteristic} of a $(d-1)$-dimensional simplicial complex $\Delta$ (where $d\geq 1$) is $$\tilde{\chi}(\Delta):=\sum_{i=-1}^{d-1} (-1)^{i} f_{i}(\Delta) =\sum_{i=0}^{d-1} (-1)^{i}\cdot\tilde{b}_i(\Delta; \field).$$

\smallskip\noindent {\it Proof of Proposition \ref{Dehn-Somm}: \ } Let $\Delta$ be a $(d-1)$-dimensional $\field$-homology manifold with boundary. Then 
\begin{equation} \label{h=h}
h_{d-i}(\Delta) + \binom{d}{i}(-1)^{d-i}\tilde{\chi}(\Delta)=h_i(\Delta)-g_i(\partial\Delta) 
= h_i(\Delta,\partial\Delta) \quad \mbox{for all } 0\leq i\leq d.
\end{equation}
Here the first equality is a result of Gr\"abe \cite{Grabe-87} (see also \cite[Theorem 3.1]{Novik-Swartz-09:Socles}), and the second one follows from Lemma \ref{h-and-g}. If, in addition, $\Delta$ is an orientable $\field$-homology manifold, then by Poincar\'e-Lefschetz duality,
$ \tilde{b}_k(\Delta; \field)=\tilde{b}_{d-1-k}(\Delta,\partial\Delta; \field)$ for all $0<k<d$. 
Hence for $i<d$,
\begin{eqnarray} \nonumber
(-1)^{d-i}\tilde{\chi}(\Delta)&=&\sum_{k=0}^{d-1} (-1)^{d-i-k}\cdot \tilde{b}_k(\Delta; \field) \\
&=& \label{b-b}
\left(\sum_{k=0}^{d-1-i} (-1)^{d-i-k}\cdot \tilde{b}_k(\Delta; \field)\right) - 
\left(\sum_{k=0}^{i-1} (-1)^{i-k}\cdot \tilde{b}_k(\Delta, \partial\Delta; \field)\right).
\end{eqnarray}
Substituting \eqref{b-b} into \eqref{h=h} and using the definition of $h''$-numbers (see eq.~\eqref{h''}) yield that for $0<i<d$,
\begin{eqnarray*}
h''_{d-i}(\Delta)&=&h_{d-i}(\Delta)+\binom{d}{i}\left(\sum_{k=0}^{d-1-i} (-1)^{d-i-k}\cdot \tilde{b}_k(\Delta; \field)\right) \\
&=&h_i(\Delta,\partial\Delta)+\binom{d}{i}\left(\sum_{k=0}^{i-1} (-1)^{i-k}\cdot \tilde{b}_k(\Delta, \partial\Delta; \field)\right)=h''_i(\Delta,\partial\Delta).
\end{eqnarray*}
The first part of the statement follows. 

Finally, as $f_{-1}(\Delta)=1$ and $f_{-1}(\Delta, \partial\Delta)=0$, we obtain that $h''_0(\Delta)=1$ and $h''_0(\Delta, \partial\Delta)=0$. On the other hand, if $\Delta$ is also connected, then $h''_d(\Delta,\partial\Delta)=\tilde{b}_{d-1}(\Delta,\partial\Delta)=1$ and $h''_d(\Delta)=\tilde{b}_{d-1}(\Delta;\field)=\tilde{b}_0(\Delta,\partial\Delta; \field)=0$, and the second part of the statement follows as well. \endproof

We also need the following result that is well-known to the experts.

\begin{lemma} \label{h''>h'}
Let $\Delta$ be an orientable $\field$-homology manifold (with or without boundary) of dimension $d-1\geq 3$. Then $h''_{d-2}(\Delta)\geq h'_{d-1}(\Delta)$. Furthermore, if all vertex links of $\Delta$ have the WLP then $h''_{d-r}(\Delta)\geq h'_{d-r+1}(\Delta)$ for all $r\leq d/2$. 
\end{lemma}
\begin{proof} For $r=2$ this is \cite[Theorem 2.6]{Swartz-14}. For other values of $r$, our assumptions on the links combined with \cite[Theorem 4.26]{Swartz-09} imply that for generic linear forms $\theta_1,\ldots, \theta_d,\omega$, the linear map $\times\omega: \left(\field[\Delta]/\Theta\field[\Delta]\right)_{d-r}\to \left(\field[\Delta]/\Theta\field[\Delta]\right)_{d-r+1}$ is surjective. Since by Schenzel's theorem (see Theorem \ref{thm:Schenzel1}), the dimensions of spaces involved are $h'_{d-r}(\Delta)$ and $h'_{d-r+1}(\Delta)$, respectively, and since the dimension of the kernel of this map is at least $\binom{d}{r}\tilde{b}_{d-r-1}(\Delta;\field)$ (see \cite[Corollary 3.6]{Novik-Swartz-09:Socles}), we conclude that $h''_{d-r}(\Delta)=h'_{d-r}(\Delta)-\binom{d}{r}\tilde{b}_{d-r-1}\geq h'_{d-r+1}(\Delta)$.
\end{proof}

We are now in a position to verify the following special case of Theorems \ref{main1} and \ref{main2}:

\begin{proposition} Let $\Delta$ be a $(d-1)$-dimensional orientable $\field$-homology manifold with nonempty boundary, where $d-1\geq 3$. Then $g_{2}(\Delta, \partial\Delta)\geq \binom{d+1}{2}\left(\tilde{b}_1(\Delta, \partial\Delta; \field)- \tilde{b}_0(\Delta, \partial\Delta; \field)\right)$. Furthermore, if all vertex links of $\Delta$ have the WLP then 
\[
g_r(\Delta, \partial\Delta)\geq \binom{d+1}{r}\sum_{j=1}^r (-1)^{r-j}\cdot\tilde{b}_{j-1}(\Delta, \partial\Delta; \field) \quad \mbox{for all } 1\leq r\leq d/2.
\]
\end{proposition}
\begin{proof} All computations below are over $\field$, and so $\field$ is omitted from the notation.
The statement is easy when $r=1$. Indeed, since $f_{-1}(\Delta,\partial\Delta)=0$, we obtain that $g_1(\Delta,\partial\Delta)=h_1(\Delta,\partial\Delta)=f_0(\Delta,\partial\Delta)$. On the other hand, $\tilde{b}_0(\Delta,\partial\Delta)$ counts the number of connected components of $\Delta$ without boundary. The result follows since each such component has at least $d+1$ vertices.

For $r\geq 2$, we have
\begin{eqnarray*}
h''_r(\Delta,\partial\Delta)&=&h''_{d-r}(\Delta) \\
&\geq & h''_{d-r+1}(\Delta)+\binom{d}{r-1}\tilde{b}_{d-r}(\Delta)\\
&=&h''_{r-1}(\Delta,\partial\Delta)+\binom{d}{r-1}\tilde{b}_{r-1}(\Delta,\partial\Delta),
\end{eqnarray*}
where the first step is by Proposition \ref{Dehn-Somm}, the second step is by Lemma \ref{h''>h'} and eq.~\eqref{h''}, and the last step is by Proposition \ref{Dehn-Somm} and Poincar\'e-Lefschetz duality. Therefore,
\begin{eqnarray*}
0&\leq& h''_r(\Delta,\partial\Delta)-h''_{r-1}(\Delta,\partial\Delta)-\binom{d}{r-1}\tilde{b}_{r-1}(\Delta,\partial\Delta) \\
&\stackrel{\mbox{\small{by eq.~\eqref{h''}}}}{=}& \left(h_r(\Delta,\partial\Delta)-\binom{d}{r}\sum_{j=1}^r (-1)^{r-j}\cdot\tilde{b}_{j-1}(\Delta,\partial\Delta)\right) \\
&& -\left(h_{r-1}(\Delta,\partial\Delta)+\binom{d}{r-1}\sum_{j=1}^r (-1)^{r-j}\cdot\tilde{b}_{j-1}(\Delta,\partial\Delta)\right)\\
&=&g_r(\Delta,\partial\Delta)-\binom{d+1}{r}\sum_{j=1}^r (-1)^{r-j}\cdot\tilde{b}_{j-1}(\Delta,\partial\Delta),
\end{eqnarray*}
and the statement follows.
\end{proof}

Although we postpone the proofs of Theorems \ref{main1} and \ref{main2} in the non-orientable case until later sections, we note that Corollary \ref{h_2-vs-int-vert} is an easy consequence of Theorems \ref{LBT-closed} and \ref{main1}. Indeed, let $\Delta$ be a $d$-dimensional normal pseudomanifold whose boundary is also a normal pseudomanifold, where $d\geq 4$. Then by Lemma \ref{h-and-g}, 
$$h_2(\Delta)=h_1(\Delta,\partial\Delta)+g_2(\Delta,\partial\Delta)+g_2(\partial\Delta)=f_0(\Delta, \partial\Delta)+g_2(\Delta,\partial\Delta)+g_2(\partial\Delta).
$$
Replacing the last two summands with their lower bounds provided by Theorems \ref{LBT-closed} and \ref{main1} (and recalling that $\dim(\Delta,\partial\Delta)=d$ while $\dim(\partial\Delta)=d-1$) completes the proof for $d\geq 4$. 

For $d=3$, we can still use Theorem \ref{main1} to bound $g_2(\Delta,\partial\Delta)$. As for $g_2(\partial\Delta)$, note that each connected component of the boundary of $\partial\Delta$ is a closed surface, so the Dehn-Sommerville relations \cite{Klee-64} tell us that $g_2(\partial\Delta)=3\big(\tilde{b}_1(\partial\Delta;\ZZ/2\ZZ)-2\tilde{b}_0(\partial\Delta;\ZZ/2\ZZ)\big)$. The result follows since $\tilde{b}_1(\partial\Delta;\ZZ/2\ZZ)\geq \tilde{b}_1(\partial\Delta;\FF)$ and $\tilde{b}_0(\partial\Delta;\ZZ/2\ZZ)= \tilde{b}_0(\partial\Delta;\FF)$.\endproof

We close this section with a remark on sharpness of Theorems \ref{main1} and \ref{main2}.

\begin{remark} \label{sharpness}
The inequalities of Theorems \ref{main1} and \ref{main2} are sharp. Indeed, it follows from \cite[Theorem 3.1 and Proposition 5.2]{Murai-Nevo-14} that for $1\leq s\leq (d+1)/2$, any closed $(s-1)$-stacked $\FF$-homology $d$-manifold satisfies
\[
g_r(\Delta) = \binom{d+2}{r}\sum_{j=1}^r (-1)^{r-j}\cdot\tilde{b}_{j-1}(\Delta; \field) \quad \mbox{for all } s\leq r\leq (d+1)/2.
\]
A straightforward computation then shows that for $1\leq s\leq (d+1)/2$, any closed 
$(s-1)$-stacked $\FF$-homology $d$-manifold with one facet removed satisfies
\[
g_r(\Delta, \partial\Delta) = \binom{d+2}{r}\sum_{j=1}^r (-1)^{r-j}\cdot\tilde{b}_{j-1}(\Delta, \partial\Delta; \field) \quad \mbox{for all } s\leq r\leq (d+1)/2.
\]
The existence of a closed $r$-stacked $d$-manifold $\Delta_{r,d}$ that triangulates $\mathbb{S}^{r}\times \mathbb{S}^{d-r}$ for all pairs $0\leq r\leq d$ was established in \cite{Klee-Novik-12} (where this manifold was denoted by $\partial \mathcal{B}(r,d+2)$). Since being $(s'-1)$-stacked implies being $(s-1)$-stacked for all $s\geq s'$ and since connected sums of $(s-1)$-stacked manifolds are $(s-1)$-stacked for all $s\geq 2$, it follows that by considering connected sums of several copies of $\Delta_{0,d},\ldots, \Delta_{s-1,d}$ we can ensure the existence of a closed $(s-1)$-stacked $d$-manifold $\Delta$ with $\tilde{b}_0(\Delta)=a_0,\tilde{b}_1(\Delta)=a_1,\ldots,\tilde{b}_{s-1}(\Delta)=a_{s-1}$ for any $2\leq s\leq (d+1)/2$ and an arbitrary non-negative integer vector $(a_0,a_1,\ldots,a_{s-1})$.
\end{remark}

\section{The $\sigma$- and $\mu$-numbers of relative simplicial complexes}
In this section we develop the theory of $\sigma$- and $\mu$-numbers for relative simplicial complexes.
All results below are natural extensions of the results proved by Bagchi and Datta \cite{Bagchi:mu-vector, Bagchi-Datta-14}. Throughout this section we do all homology computations with coefficients in a fixed field $\field$ and we omit $\field$ from our notation. For a simplicial complex $\Delta$ and a subset $W \subseteq V(\Delta)$, we denote by $\Delta_W=\{F \in \Delta: F \subseteq W\}$ the subcomplex of $\Delta$ induced by $W$. We use the following 

\begin{notation}
\label{notation1}
Let $(\Delta,\Gamma)$ be a relative simplicial complex with $V(\Delta)=V$. Fix an order $v_1,\dots,v_n$ on the elements of $V$, where $n=|V|$, and denote by $V_{\leq k}$ the set $\{v_1,\dots,v_k\}$.
\end{notation}

Note that if $\lk_\Delta(v) \ne \emptyset$, then $\lk_\Delta(v)_{V_{\leq 0}}=\{ \emptyset\}$. To define the $\sigma$- and $\mu$-numbers we need a bit of preparation.
We start by establishing a relative analogue of the inequalities which are called Morse relations in polyhedral Morse theory developed by K\"uhnel \cite{Kuhnel-90,Kuhnel-95}. (The connection to the Morse inequalities is explained in \cite[Remark 2.11]{Bagchi-Datta-14}.)

\begin{lemma} \label{1.1}
Let $(\Delta,\Gamma)$ be a relative simplicial complex. Then for $i\geq 0$, the following holds:
\begin{itemize}
\item[(i)]
$\displaystyle b_i(\Delta,\Gamma) \leq \sum_{k=1}^n \tilde b_{i-1}\big(\lk_\Delta(v_k)_{V_{\leq k-1}},\lk_\Gamma(v_k)_{V_{\leq k-1}}\big)$, and
\item[(ii)]
$\displaystyle
\sum_{j=0}^i (-1)^{i-j} b_j(\Delta,\Gamma) \leq \sum_{j=0}^i (-1)^{i-j} \sum_{k=1}^n \tilde b_{j-1} \big(\lk_\Delta(v_k)_{V_{\leq k-1}},\lk_\Gamma(v_k)_{V_{\leq k-1}}\big)$.
\end{itemize}
\end{lemma}

\begin{proof} The triple $$\Gamma_{V_{\leq k}} \subseteq \Gamma_{V_{\leq k}} \cup \Delta_{V_{\leq k-1}} \subseteq \Delta_{V_{\leq k}}$$ gives rise to the following long exact sequence in homology 
\begin{align*}
&
\longrightarrow H_i(\Gamma_{V_{\leq k}} \cup \Delta_{V_{\leq k-1}},\Gamma_{V_{\leq k}})
\longrightarrow H_i(\Delta_{V_{\leq k}},\Gamma_{V_{\leq k}})
\longrightarrow H_i(\Delta_{V_{\leq k}},\Gamma_{V_{\leq k}} \cup \Delta_{V_{\leq k-1}})\\
&\longrightarrow \quad \cdots \\
&
\longrightarrow H_0(\Gamma_{V_{\leq k}} \cup \Delta_{V_{\leq k-1}},\Gamma_{V_{\leq k}})
\longrightarrow H_0(\Delta_{V_{\leq k}},\Gamma_{V_{\leq k}})
\longrightarrow H_0(\Delta_{V_{\leq k}},\Gamma_{V_{\leq k}} \cup \Delta_{V_{\leq k-1}}) \longrightarrow 0.
\end{align*}
Note that $\Gamma_{V_{\leq k}}=\st_{\Gamma_{V_{\leq k}}}(v_k)\cup\Gamma_{V_{\leq k-1}}$ and $\Delta_{V_{\leq k}}=\st_{\Delta_{V_{\leq k}}}(v_k)\cup\Delta_{V_{\leq k-1}}$. Thus, we obtain by the excision theorem that 
\begin{equation} \label{*1}
H_i(\Gamma_{V_{\leq k}} \cup \Delta_{V_{\leq k-1}},\Gamma_{V_{\leq k}})
\cong H_i(\Delta_{V_{\leq k-1}},\Gamma_{V_{\leq k-1}}).
\end{equation}
Furthermore, for all $i\geq 0$, we have the following canonical isomorphisms of chain complexes that commute with the boundary operator:
\[
\frac{C_i\left(\Delta_{V_{\leq k}}\right)}{C_i\left(\Gamma_{V_{\leq k}}\cup \Delta_{V_{\leq k-1}}\right)}
\cong \frac{C_i\left(\st_{\Delta_{V_{\leq k}}}(v_k)\cup\Delta_{V_{\leq k-1}}\right)/C_i\left(\Delta_{V_{\leq k-1}}\right)}{C_i\left(\st_{\Gamma_{V_{\leq k}}}(v_k)\cup\Delta_{V_{\leq k-1}}\right)/C_i\left(\Delta_{V_{\leq k-1}}\right)}\cong \frac{C_{i-1}\left(\lk_{\Delta_{V_{\leq k}}}(v_k)\right)}{C_{i-1}\left(\lk_{\Gamma_{V_{\leq k}}}(v_k)\right)}.
\]
Therefore,
\begin{eqnarray}
\nonumber
H_i(\Delta_{V_{\leq k}},\Gamma_{V_{\leq k}} \cup \Delta_{V_{\leq k-1}})
&\cong&
\tilde{H}_{i-1}\big(\lk_{\Delta_{V_{\leq k}}}(v_k),\lk_{\Gamma_{V_{\leq k}}}(v_k)\big)\\
\label{*2}
&=& \tilde{H}_{i-1}\big(\lk_\Delta(v_k)_{V_{\leq k-1}},\lk_\Gamma(v_k)_{V_{\leq k-1}}\big).
\end{eqnarray}

The above long exact sequence together with \eqref{*1} and \eqref{*2} yields
\begin{align*}
b_i(\Delta_{V_{\leq k}},\Gamma_{V_{\leq k}})
& \leq b_i(\Gamma_{V_{\leq k}} \cup \Delta_{V_{\leq k-1}},\Gamma_{V_{\leq k}})
+ b_i(\Delta_{V_{\leq k}},\Gamma_{V_{\leq k}} \cup \Delta_{V_{\leq k-1}})\\
&= b_i(\Delta_{V_{\leq k-1}},\Gamma_{V_{\leq k-1}})
+ \tilde{b}_{i-1} \big(\lk_\Delta(v_k)_{V_{\leq k-1}}, \lk_\Gamma(v_k)_{V_{\leq k-1}}\big).
\end{align*}
Using this inequality inductively, we infer part (i) of the statement.

As for part (ii) of the statement, note that the long exact sequence also implies
\begin{eqnarray} 
\nonumber
&&\sum_{j=0}^i (-1)^{j-i} b_j(\Delta_{V_{\leq k}},\Gamma_{V_{\leq k}})
\\
\label{**3}
&& \leq \sum_{j=0}^i (-1)^{j-i}
b_j(\Gamma_{V_{\leq k}} \cup \Delta_{V_{\leq k-1}},\Gamma_{V_{\leq k}})
+ 
\sum_{j=0}^i (-1)^{j-i}
b_j(\Delta_{V_{\leq k}},\Gamma_{V_{\leq k}} \cup \Delta_{V_{\leq k-1}}),
\end{eqnarray}
and so the desired inequality follows similarly to that of part (i) from equations \eqref{*1} and \eqref{*2}.
\end{proof}

As a corollary of Lemma \ref{1.1}, we obtain the following result. 
\begin{proposition}
\label{1.2}
Let $(\Delta, \Gamma)$ be a relative simplicial complex with $|V(\Delta)|=n$.
Then
\begin{itemize}
\item[(i)] $\displaystyle b_i(\Delta,\Gamma) \leq \sum_{v \in V(\Delta)} \left\{ \frac{1}{n} \sum_{W \subseteq V(\Delta) \setminus\{v\}} \frac{1}{{n-1 \choose |W|}} \tilde{b}_{i-1}\big(\lk_\Delta(v)_W,\lk_\Gamma(v)_W\big)\right\}$, and
\item[(ii)]
$\displaystyle \sum_{j=0}^i (-1)^{i-j} b_j(\Delta,\Gamma) \leq \sum_{j=0}^i (-1)^{i-j}\sum_{v\in V(\Delta)} \left\{ \frac 1 n \sum_{W \subseteq V(\Delta) \setminus\{v\}} \frac 1 {{n-1 \choose |W|}} \tilde{b}_{j-1}\big(\lk_\Delta(v)_W,\lk_\Gamma(v)_W\big)\right\}$.
\end{itemize}
\end{proposition}

\begin{proof}
Let $V=V(\Delta)=\{v_1,\dots,v_n\}.$ We refer to the inequality in Lemma \ref{1.1}(i) as the Morse inequality with respect to the ordering $v_1,\dots,v_n$. Taking the sum of Morse inequalities with respect to \textbf{all} permutations of $v_1,v_2,\dots,v_n$, we obtain
$$n!\cdot b_i(\Delta,\Gamma)
\leq \sum_{v\in V} \left(\sum_{W \subseteq V\setminus\{v\}} (n-|W|-1)! (|W|)! \cdot\tilde{b}_{i-1}\big(\lk_\Delta(v)_W,\lk_\Gamma(v)_W \big) \right).$$
This is because for each $v\in V$ and $W\subseteq V\setminus\{v\}$, the set $W$ followed by $v$ shows up as an initial segment in exactly $(n-|W|-1)! (|W|)!$ permutations of $V$. Hence
$$b_i(\Delta,\Gamma)
\leq \frac 1 n \sum_{v \in V} \left(\sum_{W \subseteq V\setminus\{v\}} \frac{1}{{ n-1 \choose |W|}} \tilde{b}_{i-1}\big(\lk_\Delta(v)_W,\lk_\Gamma(v)_W\big)\right)$$
as desired. The proof of part (ii) follows from Lemma \ref{1.1}(ii) in a similar way.
\end{proof}

For a relative simplicial complex $(\Delta,\Gamma)$ on a ground set $V$, we define
$$\tilde{\sigma}_i^V(\Delta,\Gamma):= \frac{1}{|V| +1} \sum_{W \subseteq V} \frac{1}{{|V| \choose |W|}} \tilde{b}_i(\Delta_W,\Gamma_W).$$ 
The next lemma shows that $\tilde{\sigma}_i^V$ is independent of the choice of a ground set $V$. 
\begin{lemma}
\label{1.3}
With the same notation as above, if $V' \supset V$ then
$$\tilde{\sigma}_i^{V'} (\Delta,\Gamma)= \tilde{\sigma}_i^{V} (\Delta,\Gamma).$$
\end{lemma}
\begin{proof}
We may assume that $V'=V \cup \{x\}$, and so $\{x\}\notin\Delta$.
Then
\begin{align*}
\tilde{\sigma}_i^{V'}(\Delta,\Gamma)
&= \frac{1}{|V|+2} \sum_{W \subseteq V} \left\{ \frac{1}{{|V|+1 \choose |W|}} \tilde{b}_i\big(\Delta_W,\Gamma_W\big) + \frac{1}{{|V| +1 \choose |W|+1}} \tilde{b}_i\big(\Delta_{W \cup \{x\}},\Gamma_{W \cup \{x\}}\big)\right\}\\
&= \frac{1}{|V|+2} \sum_{W \subseteq V} \left\{ \frac{(|W|)!(|V|-|W|)!}{(|V|+1)!} (|V|+2) \cdot\tilde{b}_i(\Delta_W,\Gamma_W)\right\}\\
&= \frac{1}{|V|+1} \sum_{W \subseteq V} \frac{1}{{|V| \choose |W|}} \tilde{b}_i(\Delta_W,\Gamma_W)=\tilde{\sigma}_i^V(\Delta,\Gamma),
\end{align*}
as desired.
\end{proof}

In view of Lemma \ref{1.3} we make the following
\begin{definition} \label{def:sigma-and-mu} Let $(\Delta,\Gamma)$ be a relative simplicial complex with $V(\Delta)=V$. For $i\geq -1$, the \textbf{$i$-th normalized $\sigma$-number} of $(\Delta,\Gamma)$ is
$$\tilde{\sigma}_i(\Delta,\Gamma):=\tilde{\sigma}_i^V(\Delta,\Gamma)= \frac{1}{|V| +1} \sum_{W \subseteq V} \frac{1}{{|V| \choose |W|}} \tilde{b}_i(\Delta_W,\Gamma_W).$$ 
For $i\geq 0$, the \textbf{$i$-th $\mu$-number} of $(\Delta,\Gamma)$ is 
$$\mu_i(\Delta,\Gamma):=\sum_{v\in V} \tilde{\sigma}_{i-1}\big(\lk_\Delta(v),\lk_\Gamma(v)\big).$$
\end{definition}

Several remarks are in order. First we note that our definition of relative $\tilde{\sigma}$-numbers agrees with that of non-relative $\sigma$-numbers from \cite{Murai-15} (which, in turn, slightly modifies the definition given in \cite{Bagchi:mu-vector,Bagchi-Datta-14,BDS}) except for the normalizing factor of $\frac{1}{|V|+1}$, that is, $\tilde{\sigma}_i(\Delta,\emptyset)=\frac{\sigma_i(\Delta)}{|V|+1} $. (For instance, $\tilde{\sigma}_{-1}(\Delta,\Gamma)=0$ if $\Gamma\neq\emptyset$, but $\tilde{\sigma}_{-1}(\Delta,\emptyset)=1/(f_0(\Delta)+1)$ as long as $\Delta\neq\emptyset$.) At the same time, our $\mu_i(\Delta,\emptyset)$ coincides with $\mu_i(\Delta)$ from~\cite{Murai-15}.

We also observe that using Definition \ref{def:sigma-and-mu}, Proposition \ref{1.2} can be rewritten as follows.

\begin{corollary}
\label{muinequality}
For a relative simplicial complex $(\Delta,\Gamma)$, one has
\begin{itemize}
\item[(i)] $b_i(\Delta,\Gamma) \leq \mu_i(\Delta,\Gamma)$, and
\item[(ii)] $\displaystyle \sum_{j=0}^i (-1)^{i-j} b_j(\Delta,\Gamma) \leq \sum_{j=0}^i (-1)^{i-j} \mu_j(\Delta,\Gamma)$.
\end{itemize}
\end{corollary}

\noindent This result is analogous to the classical Morse inequalities and generalizes \cite[Theorem 1.8(a,c)]{Bagchi:mu-vector}.

It was proved by Bagchi and Datta that 
$\tilde \sigma_{i-1}(\Delta)=\tilde \sigma_{d-1-i}(\Delta)$
holds for any $(d-1)$-dimensional homology sphere $\Delta$ (\cite[Lemma 2.2]{Bagchi-Datta-14}) and that $\mu_i(\Delta)=\mu_{d-i}(\Delta)$ holds for any closed $d$-dimensional homology manifold $\Delta$ (\cite[Theorem 1.7]{Bagchi:mu-vector}). The following is an extension of these results to balls and manifolds with boundary, respectively.

\begin{proposition} \label{balls-duality} \quad
\begin{itemize}
\item[(i)] Let $\Delta$ be a $(d-1)$-dimensional $\field$-homology ball. Then 
$$
\tilde{\sigma}_{i-1}(\Delta,\partial\Delta)=\tilde{\sigma}_{d-1-i}(\Delta) \quad \mbox{for all } 0\leq i \leq d.$$
\item[(ii)]
Let $\Delta$ be a $d$-dimensional $\field$-homology manifold with boundary. Then $\mu_i(\Delta,\partial\Delta)=\mu_{d-i}(\Delta)$ for all $0\leq i\leq d$. 
\end{itemize}
\end{proposition}

\begin{proof} According to Definition \ref{def:sigma-and-mu}, to prove part (i), it suffices to show that for a $(d-1)$-dimensional $\field$-homology ball $\Delta$ with $V(\Delta)=V$ and for any subset $W\subseteq V$, the following duality relation holds:
\begin{equation} \label{duality}
\tilde{b}_{i-1}\big(\Delta_W, (\partial\Delta)_W\big)=\tilde{b}_{d-1-i}\big(\Delta_{V\setminus W}\big) \quad \mbox{for }0\leq i \leq d.
\end{equation}
This duality is a simple consequence of Alexander duality. Indeed, let $u$ be a vertex not in $V$, let $u\ast\partial\Delta$ be the cone over $\partial\Delta$ with apex $u$, and let $\Lambda:=\Delta\cup (u\ast\partial\Delta)$. Then $\Lambda$ is an $\field$-homology sphere and $\Lambda_{W\cup\{u\}}=\Delta_{W}\cup \big(u\ast(\partial\Delta)_{W}\big)$. 
Using the excision theorem and the fact that $u\ast(\partial\Delta)_{W}$ is contractible, we conclude that 
$$
\tilde{b}_{i-1}\big(\Delta_W, (\partial\Delta)_W\big)=\tilde{b}_{i-1}\big(\Lambda_{\{u\}\cup W}, (u\ast(\partial\Delta)_{W})\big)=\tilde{b}_{i-1}\big(\Lambda_{\{u\}\cup W}\big).
$$
Eq.~\eqref{duality}, and hence also the statement of part (i), follows since by Alexander duality 
$$
\tilde{b}_{i-1}\left (\Lambda_{\{u\}\cup W}\right) =\tilde{b}_{d-1-i}\left(\Lambda_{V\setminus W}\right)=\tilde{b}_{d-1-i}\left(\Delta_{V\setminus W}\right).$$

Now, if $\Delta$ is an $\field$-homology manifold with boundary, and $v$ is a boundary vertex of $\Delta$, then $\lk_\Delta(v)$ is an $\field$-homology ball whose boundary is given by $\lk_{\partial\Delta}(v)$. Thus part (ii) is an immediate consequence of part (i) and an analogous result for spheres proved in \cite[Lemma 2.2]{Bagchi-Datta-14}.
\end{proof}

\section{Upper bounds on graded Betti numbers}
In this section, we develop upper bounds on certain alternating sums of graded Betti numbers of the Stanley-Reisner modules of relative simplicial complexes. We are especially interested in the case of $\field[\Delta,\partial\Delta]$ where (i) $\Delta$ is a homology ball or a homology sphere with the WLP (see Theorem \ref{4.6}), or (ii) $\Delta$ is a normal pseudomanifold with boundary (Theorem \ref{4.9}).
To this end, we first establish several algebraic results.

For a graded $\field$-algebra $R$ with the maximal ideal $\mideal_R$ and a finitely-generated graded $R$-module $M$, the numbers $$\beta_{i,j}^R(M)=\dim_\field \big(\Tor_i^R(M,\field)\big)_j$$ are called the \textbf{graded Betti numbers} of $M$ over $R$; here we identify $\field$ with $R/\mideal_R$. As we will see in the next section, these numbers are closely related to the $\tilde \sigma$-numbers introduced in the previous section. 

We will mainly consider the graded Betti numbers over $S=\field[x_1,\dots,x_n]$ and will need the following two easy facts about them.

\begin{lemma}
\label{4.1} Let $M$ be a finitely generated graded $S$-module. If $M$ is generated by elements of degree $\geq j+1$, then $\beta_{i,i+\ell}^S(M)=0$ for all $i \geq 0$ and $\ell \leq j$.
\end{lemma}
\begin{proof}
The assertion holds when $i=0$ since $\beta_{0,k}^S(M)=\dim_\field(M/\mideal M)_k$ for all $k$, where $\mideal=(x_1,\dots,x_n)$. For $i>0$, the assertion follows from the facts that (i) if $M$ is generated by elements of degree $\geq \ell$ then its syzygy module, $\mathrm{Syz}(M)$, is generated by elements of degree $\geq \ell+1$, and that (ii) $\beta_{i,k}^S(\mathrm{Syz}(M))=\beta_{i+1,k}^S(M)$ for all $i \geq 0$ and $k \in \Z$.
\end{proof}

\begin{lemma}
\label{4.2}
If $M$ is a finitely generated Artinian graded $S$-module, then $\beta_{i,i+\ell}^S(M)=0$ for all $i \geq 0$ and $\ell >\max\{k:M_k \ne 0\}$.
\end{lemma}
\begin{proof}
Let $\mathbf{K}_\bullet$ be the Koszul complex with respect to the sequence $x_1,\dots,x_n$ (see \cite[\S 1.6]{BH}). Then $\Tor_i^S(M,\field)$ is isomorphic to the $i$-th homology of $\mathbf K_\bullet \otimes_S M$.
Since the module $\mathbf{K}_i \otimes_S M$ is isomorphic to the direct sum of copies of $M(-i)$, where $M(-i)$ is the graded module $M$ with grading shifted by degree $-i$, $\mathbf K_i \otimes_S M$ has no non-zero elements in degrees larger than $i+\max\{k:M_k \ne 0\}$. The statement follows.
\end{proof}

In the next proposition we study alternating sums of graded Betti numbers of the form $\sum_{k \geq 0} (-1)^k \beta_{i+k,i+\ell}^S(M)$. These sums are finite sums since $\beta_{i,j}^S(M)=0$ for $i >n$. As we will see later, these sums are related to the alternating sums of the $\mu$-numbers from Corollary \ref{muinequality}.

\begin{proposition}
\label{4.3}
Let $M$ be a finitely generated graded $S$-module and let $h_k=\dim_\field M_k$ for all $k \in \Z$.
Then
$$\sum_{k \geq 0} (-1)^k \beta_{i+k, i+\ell}^S(M) \leq \sum_{k \geq 0} (-1)^k h_{\ell-k} {n \choose i+k} \quad \mbox{ for all $i \geq 0$ and $\ell \in \Z$}.$$
Moreover, $\sum_{k \geq 0}(-1)^k \beta_{k,\ell}^S(M)= \sum_{k \geq 0} (-1)^k h_{\ell-k} {n \choose k}$ for all $\ell \in \ZZ$.
\end{proposition}
\begin{proof}
For $\ell \in \ZZ$, let $M_{\geq \ell}:=\bigoplus_{k \geq \ell}M_k$ and let $M_{\leq \ell}:=M / M_{\geq \ell+1}$. Note that $M_{\geq \ell}$ is a submodule of $M$. We first claim that
\begin{align}
\label{4-1}
\Tor_i^S(M,\FF)_{i+\ell} \cong \Tor_i^S(M_{\leq \ell+1},\FF)_{i+\ell} \quad \mbox{for all $i \geq 0$ and $\ell \in \ZZ$}.
\end{align}
Indeed, the short exact sequence
$$
0 \longrightarrow M_{\geq \ell +2} \longrightarrow M \longrightarrow M_{\leq \ell+1} \longrightarrow 0
$$
induces the following exact sequence
\begin{align*}
&\Tor_i^S(M_{\geq \ell+2},\FF)_{i+\ell} \longrightarrow \Tor_i^S(M,\FF)_{i+\ell}
\longrightarrow \Tor_i^S(M_{\leq \ell+1},\FF)_{i+\ell}\\
&\ \ \longrightarrow \Tor_{i-1}^S(M_{\geq \ell+2},\FF)_{i+\ell}.
\end{align*}
By Lemma \ref{4.1}, the head and the tail of this sequence are both zero modules, and \eqref{4-1} follows.

Next, consider the short exact sequence
\begin{align}
\label{SS}
0 \longrightarrow M_{\ell +1} \longrightarrow M_{\leq \ell+1} \longrightarrow M_{\leq \ell} \longrightarrow 0,
\end{align}
where we identify $M_{\ell+1}$ with the submodule of $M_{\leq \ell+1}$ consisting of all elements of $M_{\leq \ell+1}$ of degree $\ell+1$.
Note that $M_{\ell+1}$ is isomorphic to the direct sum of $h_{\ell+1}$ copies of $\FF(-\ell-1)$ as $S$-modules. Also, $\beta^S_{i,i}(\FF)={n \choose i}$ for all $i$, since the Koszul complex with respect to the sequence $x_1,\dots,x_n$
gives a minimal free $S$-resolution of $\FF$ (see \cite[Corollary 1.6.14]{BH}).
Thus 
\begin{align}
\label{4-2}
\beta_{i,i+\ell+1}^S(M_{\ell+1})=h_{\ell+1} {n \choose i} \quad \mbox{for all $i$}.
\end{align}

The short exact sequence \eqref{SS} induces the long exact sequence 
\begin{align*}
&0=\Tor_i^S(M_{\ell+1},\FF)_{i+\ell} \longrightarrow \Tor_i^S(M_{\leq \ell+1},\FF)_{i+\ell}
\longrightarrow \Tor_i^S(M_{\leq \ell},\FF)_{i+\ell}\\
& \stackrel {\psi_{i-1,i+\ell}} \longrightarrow\Tor_{i-1}^S(M_{\ell+1},\FF)_{i+\ell} \stackrel {\varphi_{i-1,i+\ell}} \longrightarrow \Tor_{i-1}^S(M_{\leq \ell+1},\FF)_{i+\ell}
\longrightarrow \Tor_{i-1}^S(M_{\leq \ell},\FF)_{i+\ell}=0,
\end{align*}
where the first term is zero by Lemma \ref{4.1} and the last one is zero by Lemma \ref{4.2}.
Then 
\begin{align}
\nonumber 
\beta_{i,i+\ell}^S(M_{\leq \ell+1})
&=
\beta_{i,i+\ell}^S(M_{\leq \ell})-\dim_\FF (\Image \psi_{i-1,i+\ell})\\
\label{4-3}&=
\beta_{i,i+\ell}^S(M_{\leq \ell})-\dim_\FF (\Ker \varphi_{i-1,i+\ell}) \quad\mbox{for all $i \geq 0$};
\end{align}
here $\Image \psi_{-1,\ell}$ and $\Ker \varphi_{-1,\ell}$ are zero modules. Furthermore,
$$\beta_{i-1,i+\ell}^S(M_{\leq \ell+1})=\beta_{i-1,i+\ell}^S(M_{\ell+1})-\dim_\FF (\Ker \varphi_{i-1,i+\ell}) \quad \mbox{for all $i \geq 1$}.$$
By replacing $i-1$ with $i$ and $\ell+1$ with $\ell$, the above equation can be rewritten as
\begin{align}
\label{4-4}
\beta_{i,i+\ell}^S(M_{\leq \ell})=\beta_{i,i+\ell}^S(M_{\ell})-\dim_\FF (\Ker \varphi_{i,i+\ell}) \quad \mbox{for all $i \geq 0$ and $\ell \in \ZZ$}.
\end{align}

Combining equations \eqref{4-1}, \eqref{4-3} and \eqref{4-4}, we conclude that 
\begin{equation} \label{4-6} \beta_{i,i+\ell}^S(M)
=\beta_{i,i+\ell}^S(M_{\leq \ell+1})
=\beta_{i,i+\ell}^S(M_\ell)
-\{ \dim_\FF (\Ker \varphi_{i,i+\ell})+\dim_\FF (\Ker \varphi_{i-1,i+\ell})\}
\end{equation}
for all $i \geq 0$ and $\ell \in \ZZ$.
Then
\begin{align*}
& \sum_{k \geq 0} (-1)^k \beta_{i+k,i+\ell}^S(M)\\
&\stackrel{\mbox{\small by \eqref{4-6}}}{=} \sum_{k \geq 0}(-1)^k\beta_{i+k,i+\ell}^S (M_{\ell-k}) - \sum_{k \geq 0} (-1)^k \big\{ \dim_\FF (\Ker \varphi_{i+k,i+\ell})+\dim_\FF (\Ker \varphi_{i-1+k,i+\ell})\big\}\\
&\stackrel{\mbox{\small by \eqref{4-2}}}{=} \sum_{k \geq 0} (-1)^k h_{\ell-k} {n \choose i+k} - \dim_\FF (\Ker \varphi_{i-1,i+\ell}).
\end{align*}
This proves the desired statement. The equality when $i=0$ follows since $\Ker \varphi_{-1,\ell}$ is zero.
\end{proof}

Another result we will make use of is the following lemma that appears in \cite[Corollary 8.5]{Migliore-Nagel-03}.
\begin{lemma}
\label{4.4}
Let $M$ be a finitely-generated graded $S$-module, $\theta \in S$ a linear form,
and $\ell$ an integer.
Suppose that the multiplication map $\times \theta :M_k \to M_{k+1}$
is injective for $k \leq \ell$. Then
\begin{itemize}
\item[(i)] 
$\beta_{i,i+k}^S(M) = \beta_{i,i+k}^{S/\theta S}(M/\theta M)$ if $i=1$ and $k \leq \ell$ or if $i\geq 2$ and $k \leq \ell-1$, and
\item[(ii)] $\beta_{i,i+\ell}^S(M) \leq \beta_{i,i+\ell}^{S/\theta S}(M/\theta M)$ for all $i\geq 2$.
\end{itemize}
\end{lemma}

Combining Proposition \ref{4.3} and Lemma \ref{4.4}, we obtain the following result.

\begin{lemma}
\label{4.5}
Let $(\Delta,\Gamma)$ be a relative simplicial complex with $V(\Delta)=[n]$
and let $j$ be a positive integer.
Suppose that $\Delta$ has dimension $d-1$
and that there are linear forms $\theta_1,\dots,\theta_{d+1} \in \FF[\Delta]$
such that
the multiplication map
$$\times \theta_k: \big(\FF[\Delta,\Gamma]/((\theta_1,\dots,\theta_{k-1})\FF[\Delta,\Gamma])\big)_{\ell} \to \big(\FF[\Delta,\Gamma]/((\theta_1,\dots,\theta_{k-1})\FF[\Delta,\Gamma]) \big)_{\ell+1}$$
is injective for all $\ell \leq j$ and $k=1,2,\dots,d+1$.
Then
\begin{itemize}
\item[(i)] $\dim_\FF \big(\FF[\Delta,\Gamma]/((\theta_1,\dots,\theta_{d+1})\FF[\Delta,\Gamma])\big)_\ell=g_\ell(\Delta,\Gamma)$
for all $\ell=0,1,\dots,j+1$,
\item[(ii)] 
$\displaystyle \sum_{k \geq 0} (-1)^k \beta^S_{i+k,i+\ell} (\FF[\Delta,\Gamma]) \leq \sum_{k \geq 0} (-1)^k g_{\ell-k}(\Delta,\Gamma) {n-d-1 \choose i+k}
\mbox{ for all $i \geq 1$ and $\ell \leq j$,}$ and
\item[(iii)] 
$\displaystyle \sum_{k \geq 0} (-1)^k \beta^S_{k,\ell} (\FF[\Delta,\Gamma]) = \sum_{k \geq 0} (-1)^k g_{\ell-k}(\Delta,\Gamma) {n-d-1 \choose k}$ for $\ell \leq j+1$.
\end{itemize}
\end{lemma}

\begin{proof}
To simplify the notation, we write $\Theta=(\theta_1,\dots,\theta_{d+1})$, $R=\FF[\Delta]$, and $M=\FF[\Delta,\Gamma]$. The Hilbert series of $M$ can be written in the form
$$\sum_{k \geq 0} (\dim_\FF M_k)t^k = \frac {\sum_{i=0}^d h_i (\Delta,\Gamma)t^i} {(1-t)^d}$$
(see \cite[III Proposition 7.1]{Stanley-96}).
Also, by the assumption, for $1\leq k\leq d$ and $\ell \leq j+1$,
$$\dim_\FF \big(M/((\theta_1,\dots,\theta_{k+1})M)\big)_\ell = \dim_\FF \big(M/((\theta_1,\dots,\theta_{k})M)\big)_\ell - \dim_\FF \big(M/((\theta_1,\dots,\theta_{k})M)\big)_{\ell-1}$$
holds. These facts then easily imply part (i) exactly as in the non-relative case.

We now turn to part (ii). By Lemma \ref{4.4}, 
\begin{align}\label{4.5-1}\sum_{k \geq 0}(-1)^k \beta_{i+k,i+\ell}^S(M)
\leq \sum_{k \geq 0} (-1)^k \beta_{i+k,i+\ell}^{S/(\Theta S)}(M/\Theta M) \quad \mbox{for all $i \geq 1$ and $\ell \leq j$}.\end{align}
Since $S/(\Theta S)$ is isomorphic to $\FF[x_1,\dots,x_{n-d-1}]$
as a ring, Proposition \ref{4.3} yields that
\begin{align}\label{4.5-2}
\sum_{k \geq 0}(-1)^k \beta_{i+k,i+\ell}^{S/(\Theta S)}(M/\Theta M)
&\leq \sum_{k \geq 0} (-1)^k \cdot \dim_\FF (M/(\Theta M))_{\ell-k} \cdot {n-d-1 \choose i+k}\\
\nonumber&= \sum_{k \geq 0} (-1)^k g_{\ell-k} (\Delta,\Gamma) {n-d-1 \choose i+k} \quad \mbox{for all $i \geq 0$ and $\ell \leq j+1$},
\end{align}
proving (ii).

Finally we prove (iii).
Observe that $\beta_{0,k}^S(M)=\dim_\FF(M/\mideal M)_k=\beta_{0,k}^{S/(\Theta S)}(M/\Theta M)$ for all $k \in \ZZ$. Since Lemma \ref{4.4} says that we have equality in \eqref{4.5-1} when $i=1$ and Proposition \ref{4.3} says that we have equality in \eqref{4.5-2} when $i=0$, it follows that
$$\sum_{k \geq 0} (-1)^k \beta_{k,\ell}^S(M)= \sum_{k\geq 0}(-1)^k \beta_{k,\ell}^{S/(\Theta S)} (M/\Theta M) = \sum_{k \geq 0}(-1)^k g_{\ell-k}(\Delta,\Gamma){n-d-1 \choose k}$$
for all $\ell \leq j+1$.
\end{proof}

We are now in a position to prove two main results of this section. We do this by applying Lemma \ref{4.5} to two combinatorial situations: in the first one, $\Delta$ is a homology ball or a homology sphere with the WLP, and in the second one, $\Delta$ is a normal pseudomanifold with boundary.

\begin{theorem}
\label{4.6}
Let $\Delta$ be an $\FF$-homology ball or an $\FF$-homology sphere of dimension $d-1$ with $V(\Delta)=[n]$.
If $\Delta$ has the WLP over $\FF$, then for all $i \geq 0$ and $\ell \leq (d-1)/2$,
$$\sum_{k \geq 0} (-1)^k \beta^S_{i+k,i+\ell}\big(\FF[\Delta,\partial \Delta]\big) \leq \sum_{k \geq 0} (-1)^k g_{\ell-k}(\Delta,\partial \Delta){n-d-1 \choose i+k}.$$
Moreover, if $\sum_{k \geq 0} (-1)^k \beta^S_{1+k,1+\ell}(\FF[\Delta]) = \sum_{k \geq 0} (-1)^k g_{\ell-k}(\Delta){n-d-1 \choose 1+k}$ for some $0 \leq \ell \leq (d-1)/2$ and if $\Delta$ is an $\field$-homology sphere, then $g_{\ell+1}(\Delta)=0$.
\end{theorem}

\begin{proof}
Let $R=\FF[\Delta]$ and $M=\FF[\Delta,\partial \Delta]$.
Since $\Delta$ has the WLP, there is an l.s.o.p.~$\theta_1,\dots,\theta_{d}$ of $\FF[\Delta]$ and a linear form $\theta_{d+1}$ such that
\begin{align}
\label{4-8}
\times \theta_{d+1} : (R/(\Theta R))_{d-\ell-1} \to (R/(\Theta R))_{d-\ell}
\quad \mbox{is surjective for $\ell \leq (d-1)/2$,}
\end{align}
where $\Theta=(\theta_1,\dots,\theta_{d})$. On the other hand, since $\Delta$ is a homology ball or a homology sphere,
$M=\FF[\Delta,\partial \Delta]$ is the canonical module of $R$ \cite[II Theorem 7.3]{Stanley-96}.
(Note that if $\Delta$ is a homology sphere, then $\partial \Delta=\emptyset$
and the canonical module of $R$ is $R$ itself.)
Thus $\theta_1,\dots,\theta_{d}$ is also an l.s.o.p.~of $M$
and $M/(\Theta M)(+d)$ is isomorphic to the Matlis dual of $R/(\Theta R)$
(see e.g.\ \cite[Lemma 3.6]{MY}).
The surjectivity in \eqref{4-8} then implies that
$$
\times \theta_{d+1} : \big(M/(\Theta M)\big)_{\ell}
\to \big(M/(\Theta M)\big)_{\ell+1}
$$
is injective for $\ell \leq \frac {d-1} 2$.
Since $\Theta$ is a regular sequence of $M$,
by applying Lemma \ref{4.5} to $\theta_1,\dots,\theta_{d+1}$ and $(\Delta,\partial \Delta)$, we obtain the desired inequality.

Now, suppose that $\Delta$ is an $\FF$-homology sphere and that for some $0 \leq \ell \leq (d-1)/2 $, one has $\sum_{k \geq 0} (-1)^k \beta^S_{1+k,1+\ell}(\FF[\Delta]) = \sum_{k \geq 0} (-1)^k g_{\ell-k}(\Delta){n-d-1 \choose 1+k}$. Then $g_{\ell+1}(\Delta)=0$, since according to Lemma \ref{4.5}(iii), 
$$ \beta_{0,1+\ell}^S(\FF[\Delta])-\sum_{k \geq 0} (-1)^k \beta^S_{1+k,1+\ell}(\FF[\Delta])
= g_{\ell+1}(\Delta)- \sum_{k \geq 0} (-1)^k g_{\ell-k}(\Delta){n-d-1 \choose 1+k}$$
and since $\beta_{0,1+\ell}^S(\FF[\Delta])=0$ for $\ell \ge 0$.
\end{proof}

The second main result of this section, Theorem \ref{4.9}, concerns normal pseudomanifolds.
We say that a $(d-1)$-dimensional pure simplicial complex is a \textbf{minimal $(d-1)$-cycle complex} if, for some field $\field$, there is a \emph{cycle} $\sum_{ G } \alpha_G e_G \in C_{d-1}(\Delta)$ such that (i) $\alpha_G \in \FF$ is non-zero for every facet $G \in \Delta$, and (ii) for each proper subset $\Gamma \subset \{G \in \Delta:|G|=d\}$, the sum $\sum_{G \in \Gamma} \alpha_G e_G$ is \emph{not} a cycle.
The following result was essentially proved in Fogelsanger's thesis \cite{Fogelsanger-88}.

\begin{theorem}
\label{4.7}
Let $\Delta$ be a minimal $(d-1)$-cycle complex and let $\FF$ be any infinite field.
If $d \geq 3$, then for a generic choice of linear forms $\theta_1,\dots,\theta_{d+1} \in \FF[\Delta]$,
the multiplication map
$$\times \theta_k : \big(\FF[\Delta]/((\theta_1,\dots,\theta_{k-1})\FF[\Delta]) \big)_1 \to \big(\FF[\Delta]/((\theta_1,\dots,\theta_{k-1})\FF[\Delta])\big)_2$$
is injective for $k=1,2,\dots,d+1$.
\end{theorem}

Fogelsanger actually proved that every minimal $(d-1)$-cycle complex has a generically $d$-rigid $1$-skeleton. In characteristic zero, Theorem \ref{4.7} is equivalent to this result of Fogelsanger by the work of Lee \cite{Lee-94}. For non-zero characteristic, the statement follows since, as was shown in \cite{Novik-Swartz-09:DS} (see the discussion and references in \cite[\S 5]{Novik-Swartz-09:DS}), the methods used in Fogelsanger's thesis \cite{Fogelsanger-88} provide a characteristic independent proof of the theorem.
Theorem \ref{4.7} leads to the following statement about relative simplicial complexes.
\begin{lemma}
\label{4.8}
Let $\Delta$ be a minimal $(d-1)$-cycle complex, $\Gamma$ a subcomplex of $\Delta$, and $\FF$ an infinite field.
If $d \geq 3$ and $f_0(\Gamma) \geq d$, then there are linear forms $\theta_1,\dots,\theta_{d+1} \in \FF[\Delta]$
such that the multiplication map
$$\times \theta_k : \big(\FF[\Delta,\Gamma]/((\theta_1,\dots,\theta_{k-1})\FF[\Delta,\Gamma])\big)_1 \to \big(\FF[\Delta,\Gamma]/((\theta_1,\dots,\theta_{k-1})\FF[\Delta,\Gamma])\big)_2$$
is injective for $k=1,2,\dots,d+1$.
\end{lemma}

\begin{proof}
Let $R=\FF[\Delta]$ and $J=\FF[\Delta,\Gamma]$.
Then $J$ is an ideal of $R$ and by our assumptions
\begin{align}
\label{4-7}
\dim_\FF J_1 = f_0(\Delta) -f_0(\Gamma) \leq f_0(\Delta)-d = \dim_\FF R_1 -d.
\end{align}
Hence, for a generic choice of linear forms $\theta_1,\dots,\theta_{d+1} \in \FF[\Delta]$,
\begin{itemize}
\item[(a)] the sequence
$\theta_1,\dots,\theta_{d+1}$ satisfies the conclusions of Theorem \ref{4.7} for $\FF[\Delta]$, and
\item[(b)] the natural map $J_1 \to (R/(\theta_1,\dots,\theta_k)R)_1$ induced by the inclusion $J \subset R$ is injective for $k=1,2,\dots,d$. (Note that we use \eqref{4-7} to derive this property.)
\end{itemize}

Consider the following commutative diagram
\begin{eqnarray*}
\begin{array}{cccccc}
\big(J/(\theta_1,\dots,\theta_{k-1})J\big)_2 & \longrightarrow & \big(R/(\theta_1,\dots,\theta_{k-1})R\big)_2\medskip\\
\times \theta_{k} \mbox{\LARGE{$\uparrow$}}\hspace{20pt} & & \hspace{20pt} \mbox{\LARGE{$\uparrow$}}\times \theta_k \smallskip\\
J_1=\big(J/(\theta_1,\dots,\theta_{k-1})J\big)_1 & \longrightarrow & \big(R/(\theta_1,\dots,\theta_{k-1})R\big)_1.
\end{array}
\end{eqnarray*}
By properties (a) and (b), the right vertical map and the lower horizontal map are injective for $k=1,2,\dots,d+1$.
Hence the left vertical map is also injective.
\end{proof}

We are now ready to state and prove the second main result of this section.
\begin{theorem}
\label{4.9}
Let $\Delta$ be a $(d-1)$-dimensional normal pseudomanifold with boundary, where $d\geq 3$.
Suppose $V(\Delta)=[n]$. Then
$\displaystyle\beta_{i,i+1}^S (\FF[\Delta,\partial \Delta]) \leq g_1(\Delta,\partial \Delta) {n-d-1 \choose i}$ for $i \geq 0$.
\end{theorem}
\begin{proof} Let $u$ be any element not in $[n]$, let $\Lambda:=\Delta \cup (u*\partial \Delta)$, and let $\Gamma:=u*\partial \Delta$. Then $(\Lambda,\Gamma)=(\Delta,\partial\Delta)$ as relative simplicial complexes. Furthermore, $\Lambda$ is a minimal $(d-1)$-cycle complex and $\Gamma$ has at least $d$-vertices. Thus, by Lemma \ref{4.8}, there exist $\theta_1,\ldots,\theta_{d+1}\in \FF[\Lambda]$ such that the multiplication maps $\times \theta_k : \big(\FF[\Lambda,\Gamma]/((\theta_1,\dots,\theta_{k-1})\FF[\Lambda,\Gamma])\big)_1 \to \big(\FF[\Lambda,\Gamma]/((\theta_1,\dots,\theta_{k-1})\FF[\Lambda,\Gamma])\big)_2$
are injective for $k=1,2,\dots,d+1$. On the other hand, the variable corresponding to $u$ annihilates $\field[\Lambda,\Gamma]$, and $\FF[\Lambda,\Gamma]$ is isomorphic to $\FF[\Delta,\partial\Delta]$ as an $\FF[\Delta]$-module. Thus the natural images of $\theta_1,\ldots,\theta_{d+1}$ in $\FF[\Delta]$ provide the set of $d+1$ linear forms that satisfy the conclusions of Lemma \ref{4.8} with respect to $\FF[\Delta,\partial\Delta]$, and so applying Lemma \ref{4.5} to $(\Delta,\partial\Delta)$ completes the proof. (Note that since $\partial\Delta\neq \emptyset$, $\beta_{i+k,i+1}^S(\FF[\Delta,\partial\Delta])=0$ and $g_{1-k}(\Delta,\partial\Delta)=0$ for $k >0$. This holds since $f_{-1}(\Delta,\partial\Delta)=0$ and $\FF[\Delta,\partial\Delta]_0=0$.)
\end{proof}

\section{Manifolds whose vertex links have the WLP}
The goal of this section is to prove (a strengthening of) Theorem \ref{main2}.
We start by recalling Hochster's formula, which expresses the graded Betti numbers of Stanley-Reisner modules in terms of topological Betti numbers of simplicial complexes.

\begin{theorem}[Hochster's formula]
\label{5.1}
Let $(\Delta,\Gamma)$ be a relative simplicial complex with $V(\Delta)=[n]$.
Then
$$\beta_{i,i+j}^S (\FF[\Delta,\Gamma]) = \sum_{W \subseteq [n],\ \! |W|=i+j} \tilde b_{j-1}(\Delta_W,\Gamma_W;\FF) \quad \mbox{ for all $i\geq 0$ and $j \geq 0$}.$$
\end{theorem}

\begin{proof}
This result is well-known in commutative algebra.
However, we sketch its proof since we could not find a reference to the relative version.

The ring $S=\FF[x_1,\ldots,x_n]$ has a natural $\ZZ^n$-grading defined by $\deg x_i=\mathbf e_i$ for all $i\in[n]$, where $\mathbf e_i$ is the $i$-th unit vector of $\ZZ^n$. For $F \subseteq [n]$, let $\mathbf e_F:=\sum_{i \in F} \mathbf e_i$ and $x_F:=\prod_{i \in F} x_i$. Let $\mathbf K_\bullet$ be the Koszul complex with respect to $x_1,\dots,x_n$ and let $M=\FF[\Delta,\Gamma]$.
Then $\Tor_i^S(M,\field)$ is isomorphic to $H_i (\mathbf K_\bullet \otimes_S M)$.
Since $M$ is a squarefree module, it follows from \cite[Corollary 2.4]{Ya} that
the graded Betti numbers of $M$ are concentrated in squarefree degrees, that is,
$$\beta_{i,k}^S(M)=\sum_{W \subseteq [n],\ \! |W|=k} \dim_\FF\Tor_i^S(M,\FF)_{\mathbf e_W}.$$
To prove the theorem, we show that the complex $(\mathbf K_\bullet \otimes_S M)_{\mathbf e_W}$ is isomorphic to the simplicial cochain complex of $(\Delta_W,\Gamma_W)$ with an appropriate shift of homological positions.
Indeed, 
$$(\mathbf K_i \otimes_S M)_{\mathbf e_W}
\cong \bigoplus_{F \subseteq [n],\ \! |F| =i}S(-\mathbf e_F)\otimes_S M_{\mathbf e_{W}-\mathbf e_F}
\cong \bigoplus_{F \subseteq W,\ \! |F|=i} S(-\mathbf e_F)\otimes_S M_{\mathbf e_{W\setminus F}}
$$
has an $\FF$-basis $\{ 1_F\otimes_S x_{W\setminus F}: W\setminus F \in (\Delta,\Gamma)\}$,
where $1_F$ is a unit element of $S(-\mathbf e_F)$.
By identifying $1_F\otimes_S x_{W\setminus F}$ with the face $W\setminus F \in (\Delta_W,\Gamma_W)$,
one can easily verify that the complex $(\mathbf K_\bullet \otimes_S M)_{\mathbf e_W}$
is isomorphic to the simplicial cochain complex of $(\Delta_W,\Gamma_W)$, and so $H_i((\mathbf K_\bullet \otimes _S M))_{\mathbf e_W} \cong H_{|W|-i-1}(\Delta_W,\Gamma_W)$ for all $i$. The statement follows.
\end{proof}

By Hochster's formula, the $\tilde \sigma$-numbers of $(\Delta,\Gamma)$ introduced in Definition \ref{def:sigma-and-mu} can be rewritten in terms of the graded Betti numbers as follows. (Here we assume that $V(\Delta)=[n]$.)
\begin{align}
\label{5-1}
\tilde \sigma_{i-1}(\Delta,\Gamma)= \frac 1 {n+1} \sum_{k=0}^n \frac 1 {{n \choose k}} \beta_{k-i,k}^S (\FF[\Delta,\Gamma]).
\end{align}
This formula and Theorem \ref{4.6} lead to the following upper bounds on the alternating sums of $\tilde \sigma$-numbers.

\begin{proposition}
\label{5.2}
Let $\Delta$ be an $\FF$-homology ball or an $\FF$-homology sphere of dimension $d-1$. If $\Delta$ has the WLP, then
$$\sum_{i=0}^j (-1)^{j-i} \tilde \sigma_{i-1}(\Delta,\partial \Delta) \leq
\frac 1 {d+2} \sum_{i=0}^j (-1)^{j-i} \frac {g_i(\Delta,\partial \Delta)} {{d+1 \choose i}}
\quad \mbox{for all $j \leq (d-1)/2$}.$$
Moreover, if equality holds for some $j\leq (d-1)/2$ and if $\Delta$ is an $\FF$-homology sphere, then $g_{j+1}(\Delta)=0$.
\end{proposition}

\begin{proof}
Suppose $V(\Delta)=[n]$ and fix $j\leq(d-1)/2$.
By \eqref{5-1} and Theorem \ref{4.6},
\begin{align*}
&\sum_{i=0}^j (-1)^{j-i} \tilde \sigma_{i-1}(\Delta,\partial \Delta)\\
&=
\frac 1 {n+1} \sum_{k=0}^n \frac 1 {{n \choose k}} \left\{ \sum_{i=0}^j (-1)^{j-i} \beta_{k-i,k}^S\big(\FF[\Delta,\partial \Delta]\big)\right\}\\
&=
\frac 1 {n+1} \sum_{k=0}^n \frac 1 {{n \choose k}} \left\{ \sum_{\ell\geq 0} (-1)^\ell \beta_{k-j+\ell,k}^S\big(\FF[\Delta,\partial \Delta]\big)\right\}\\
&\leq
\frac 1 {n+1} \sum_{k=0}^n \frac 1 {{n \choose k}} \left\{ \sum_{\ell\geq 0} (-1)^\ell
g_{j-\ell}(\Delta,\partial \Delta) {n-d-1 \choose k-j+\ell} \right\}\\
&=
\sum_{\ell=0}^j (-1)^\ell g_{j-\ell}(\Delta, \partial \Delta)
\left\{ \sum_{k=0}^n \frac 1 {(n+1){n \choose k}} 
{n-d-1 \choose k-j+\ell} \right\},
\end{align*}
where we use the fact that $\beta_{k-j+\ell,k}(\FF[\Delta,\partial \Delta])=0$ when $k-j+\ell >k$ for the second equality. (We also use the convention that ${a \choose b}=0$ if $b<0$.)
Also, if equality holds and if $\Delta$ is an $\FF$-homology sphere,
then $g_{j+1}(\Delta)=0$ by Theorem \ref{4.6}. The proposition then follows from the simple combinatorial identity discussed in the next lemma.
\end{proof}

\begin{lemma}
\label{5.3}
Let $n \geq d+1 \geq r \geq 0$ be integers. Then
$$\sum_{k=0}^n \frac 1 {(n+1) {n \choose k}} {n-d-1 \choose k-r} = \frac 1 {(d+2){d+1 \choose r}}.$$. 
\end{lemma}

\begin{proof} First note that
\begin{align*}
&\big\{ F \subseteq [n+1]: |F| = d+2\big\}\\
&= \biguplus_{k=0}^n \big\{ F \cup \{k+1\} \cup G \ : \ |F|=r, \ |G|=d+1-r, \
\max F<k+1<\min G\big\}.
\end{align*}
Thus,
\begin{equation} \label{set-part}
\sum_{k=0}^n {k \choose r} {n-k \choose d+1 -r} ={n+1 \choose d+2},
\end{equation}
and we obtain the desired equation from the following calculation:
\begin{align*}
& \sum_{k=0}^n \frac 1 {(n+1){n \choose k}} {n-d-1 \choose k-r}\\
&=
\frac {(d+2)} {(d+2){d+1 \choose r}} \sum_{k=0}^n \frac {k! (n-k)!} { (n+1)!} 
\frac {(n-d-1)!} {(k-r)! (n-d-1-k+r)!} \frac {(d+1)!} { r! (d+1 -r)! }\\
&=
\frac {1} {(d+2){d+1 \choose r}} \sum_{k=0}^n {k \choose r} \frac 1 {{n+1 \choose d+2}} {n-k \choose d+1 -r}\\
&=
\frac {1} {(d+2){d+1 \choose r}} \left\{\frac 1 {{n+1 \choose d+2}} 
\sum_{k=0}^n {k \choose r} {n-k \choose d+1 -r} \right\}\\
&=
\frac {1} {(d+2){d+1 \choose r}},
\end{align*}
where the last step is by equation \eqref{set-part}.
\end{proof}

We also need the following fact that can be proved in the same way as \cite[Proposition~2.3]{Swartz-04/05}. (For the boundary complexes of simplicial polytopes this result goes back to McMullen, see \cite[p.~183]{McMullen-70}.)

\begin{lemma}
\label{5.4}
Let $(\Delta,\Gamma)$ be a relative simplicial complex.
If $\Delta$ is pure of dimension $d$, then
$$\sum_{v \in V(\Delta)} g_k\big(\lk_\Delta(v),\lk_\Gamma(v)\big)=(d+2-k)g_k(\Delta,\Gamma)+(k+1)g_{k+1}(\Delta,\Gamma) \quad \mbox{for all $k\geq 0$}.$$
\end{lemma}

\begin{proof}
Note that
\begin{align}
\label{linksum}
\sum_{v \in V(\Delta)} f_{i-1}\big(\lk_\Delta(v),\lk_\Gamma(v)\big)=(i+1) f_i(\Delta,\Gamma)
\quad \mbox{ for all } i \geq 0.
\end{align}
The assertion of the lemma then follows from this observation by a routine computation exactly in the same way 
as its non-relative version, see \cite[Proposition 2.3]{Swartz-04/05}.
\end{proof}

We are now in a position to prove Theorem \ref{main2}.
Since $\tilde b_0(\Delta,\Gamma)=b_0(\Delta,\Gamma)-g_0(\Delta,\Gamma)$
if $\dim\Delta \geq 0$,
the following result and the Morse inequalities of Corollary \ref{muinequality} imply Theorem \ref{main2}.
Thus the following result can be seen as a strengthening of Theorem \ref{main2}.

\begin{theorem}
\label{5.5}
Let $\Delta$ be a $d$-dimensional $\FF$-homology manifold with or without boundary.
If all vertex links of $\Delta$ have the WLP, then
$$g_r(\Delta,\partial \Delta) \geq {d+2 \choose r} \left\{ \sum_{k=1}^r (-1)^{r-k} \mu_{k-1}(\Delta,\partial \Delta) + (-1)^r g_0(\Delta,\partial \Delta)\right\}\;\; \mbox{ for all }r \leq (d+1)/2.$$
Moreover, if equality holds for some $r \leq d/2$, then the link of each interior vertex is an $(r-1)$-stacked $\field$-homology sphere.
\end{theorem}

\begin{proof}
The proof is similar to that of \cite[Theorem 3.6]{Bagchi-Datta-14}.
By Proposition \ref{5.2}, for a fixed $r\leq (d+1)/2$,
\begin{align*}
&\sum_{k=1}^r (-1)^{r-k} \mu_{k-1}(\Delta,\partial\Delta)\\
&= \sum_{k=1}^r \sum_{v \in V(\Delta)} (-1)^{r-k} \tilde \sigma_{k-2}\big(\lk_\Delta(v),\partial(\lk_\Delta(v))\big)\\
&\leq 
\sum_{k=1}^r (-1)^{r-k} \frac { \sum_{v \in V(\Delta)} g_{k-1} \big(\lk_\Delta(v),\partial(\lk_\Delta(v))\big)} {(d+2){d+1 \choose k-1}}\\
&\stackrel {(\star)} {=} \sum_{k=1}^r (-1)^{r-k} \frac {\{ (d+3 -k) g_{k-1}(\Delta,\partial \Delta) + k g_{k}(\Delta,\partial \Delta)\} }
{(d+2) {d+1 \choose k-1}}\\
&=\sum_{k=1}^r (-1)^{r-k} \left\{
\frac {g_{k-1}(\Delta,\partial \Delta)} {{d+2 \choose k-1}} + \frac {g_{k}(\Delta,\partial \Delta) } {{d+2 \choose k}}\right\}\\
&= \frac {g_r(\Delta,\partial \Delta)} {{d+2 \choose r}}
+(-1)^{r-1} g_0(\Delta,\partial \Delta),
\end{align*}
where we use Lemma \ref{5.4} for step ($\star$).
The equality statement (for $r\leq d/2$) now follows from Theorem~\ref{GLBTsphere} and the equality statement of Proposition \ref{5.2}.
\end{proof}

\begin{remark} \label{remark-new}
Proposition \ref{5.2} and Theorem \ref{5.5} generalize the results proved by Bagchi and Datta in \cite{Bagchi:mu-vector,Bagchi-Datta-14}. They are new not only for homology balls and homology manifolds with boundary, respectively, but also for homology spheres and homology manifolds without boundary, respectively. Indeed, these results provide partial affirmative answers to \cite[Conjectures 1, 2, and 3]{Bagchi:mu-vector}. For instance, Proposition \ref{5.2} verifies the inequality part of \cite[Conjecture 1]{Bagchi:mu-vector} in the special case of homology spheres that have the WLP.
\end{remark}

There are two known large classes of $\FF$-homology spheres that have the WLP.
One such class consists of the boundary complexes of simplicial polytopes; these spheres have the WLP over $\mathbb Q$, see \cite{Stanley-80}. The other class is that of $r$-stacked $\FF$-homology spheres of dimension $\geq 2r-1$; these complexes have the WLP over $\FF$ by a result of Swartz, see \cite[Corollary 6.3]{Swartz-14}. Therefore, as a corollary of Proposition \ref{5.2}, we obtain the following result that answers \cite[Question 3.17]{Bagchi-Datta-14}.
\begin{corollary}
Let $\Delta$ be an $r$-stacked $\FF$-homology sphere of dimension $d-1\geq 2r-1$. Then
$\sum_{i=0}^j (-1)^{j-i} \tilde \sigma_{i-1}(\Delta) \leq \frac 1 {d+2} \sum_{i=0}^j (-1)^{j-i} \frac {g_i(\Delta)} {{d+1 \choose i}}$ for all $0\leq j \leq (d-1)/2$.
\end{corollary}


\section{The Lower Bound Theorem for normal pseudomanifolds}

In this section, we establish Theorem \ref{main1} in its full generality. One can prove the inequality part of this theorem in the same way as the inequality of Theorem \ref{5.5}. However, we slightly change the formulation to help our discussion of the equality case. We first provide an upper bound on $\tilde \sigma_0(\Delta,\partial\Delta)$ when $\Delta$ is a normal pseudomanifold with boundary.

\begin{proposition}
\label{6.1}
Let $\Delta$ be a $(d-1)$-dimensional normal pseudomanifold with nonempty boundary, and assume $|V(\Delta)|=n$.
If $d \geq 3$, then
$${d+2 \choose 2} \tilde \sigma_0(\Delta,\partial \Delta) \leq \frac 1 2 f_0(\Delta,\partial \Delta).$$
Moreover, if ${d+2 \choose 2} \tilde \sigma_0(\Delta,\partial \Delta) = \frac 1 2 f_0(\Delta,\partial \Delta)$, then $\beta_{n-d-1,n-d}^S (\FF[\Delta,\partial \Delta])=f_0(\Delta,\partial \Delta)$.
\end{proposition}

\begin{proof}
By \eqref{5-1} and Theorem \ref{4.9},
$$\tilde \sigma_0(\Delta,\partial \Delta)= \frac 1 {n+1} \sum_{k=0}^n \frac 1 {{n \choose k}} \beta^S_{k-1,k}(\FF[\Delta,\partial \Delta]) \stackrel{(*)}{\leq} g_1(\Delta,\partial \Delta) \sum_{k=0}^n \left\{ \frac 1 {(n+1) {n \choose k}} {n-d-1 \choose k-1} \right\},$$ and if $(*)$ holds as equality, then $\beta_{n-d-1,n-d}^S (\FF[\Delta,\partial \Delta])=g_1(\Delta,\partial \Delta)$.
Since $g_1(\Delta,\partial \Delta)=f_0(\Delta,\partial \Delta)$ and since $\sum_{k=0}^n \frac 1 {(n+1){n \choose k}} {n-d-1 \choose k-1}= \frac 1 {2 {d+2 \choose 2}}$ (see Lemma \ref{5.3}), the above inequality yields the desired statement.
\end{proof}

The next result essentially appeared in the proof of \cite[Theorem 5.3]{Murai-15}.

\begin{lemma}
\label{6.2}
Let $\Delta$ be a $(d-1)$-dimensional normal pseudomanifold without boundary.
If $d \geq 3$, then
$${d+2 \choose 2} \big(\tilde \sigma_0(\Delta) - \tilde \sigma_{-1}(\Delta)\big)
\leq \frac 1 2 f_0(\Delta) - (d+1).$$
Moreover, ${d+2 \choose 2} \big(\tilde \sigma_0(\Delta) - \tilde \sigma_{-1}(\Delta)\big)
= \frac 1 2 f_0(\Delta) - (d+1)$ if and only if $\Delta$ is a stacked sphere.
\end{lemma}

\begin{proof}
By \cite[Corollary 5.8]{Murai-15},
$${d+2 \choose 2} \tilde \sigma_0(\Delta) \leq \frac 1 {f_0(\Delta)+1} {f_0(\Delta)-d \choose 2},$$
and equality holds if and only if $\Delta$ is a stacked sphere.
Since $\tilde \sigma_{-1}(\Delta)=\frac 1 {f_0(\Delta)+1}$, it follows that
$${d+2 \choose 2} \big(\tilde \sigma_0(\Delta) -\tilde \sigma_{-1}(\Delta) \big) \leq \frac 1 {f_0(\Delta)+1} \left\{ {f_0(\Delta)-d \choose 2}-{d+2 \choose 2} \right\} = \frac 1 2 f_0(\Delta)-(d+1).$$
This proves the desired statement.
\end{proof}

We are now ready to verify the inequality part of Theorem \ref{main1}. In fact, we prove the following stronger statement. (It implies the inequality part of Theorem \ref{main1} by Corollary \ref{muinequality}.)

\begin{theorem}
\label{6.3}
Let $\Delta$ be a $d$-dimensional normal pseudomanifold with nonempty boundary.
If $d \geq 3$, then
$$g_2(\Delta,\partial \Delta) \geq {d+2 \choose 2} \big(\mu_1(\Delta,\partial \Delta;\FF)-\mu_0(\Delta,\partial \Delta;\FF)\big).$$
Furthermore, $g_2(\Delta,\partial \Delta) = {d+2 \choose 2} \big(\mu_1(\Delta,\partial \Delta;\FF)-\mu_0(\Delta,\partial \Delta;\FF)\big)$ if and only if for every boundary vertex $v$, ${d+2 \choose 2} \tilde \sigma_0\big(\lk_\Delta(v),\partial \lk_\Delta(v)\big) = \frac 1 2 f_0\big(\lk_\Delta(v),\partial \lk_\Delta(v)\big)$, and for every interior vertex $v$, the link of $v$ is a stacked sphere.
\end{theorem}

\begin{proof}
Observe that $\tilde \sigma_{-1}(\lk_\Delta(v),\lk_{\partial \Delta}(v))=f_{-1}(\lk_\Delta(v),\lk_{\partial \Delta}(v))=0$ if $\{ v\} \in \partial \Delta$.
Then, by Proposition \ref{6.1} and Lemma \ref{6.2}, 
\begin{align*}
&{d+2 \choose 2} \big\{\mu_1(\Delta,\partial \Delta)-\mu_0(\Delta,\partial \Delta)\big\}\\
&= {d+2 \choose 2} \sum_{v \in V(\Delta)} \left\{\tilde \sigma_0\big(\lk_\Delta(v),\lk_{\partial \Delta}(v)\big)-\tilde \sigma_{-1}\big(\lk_\Delta(v),\lk_{\partial \Delta}(v)\big)\right\}\\
& \leq \sum_{v \in V(\Delta)} \left\{ \frac 1 2 f_0\big(\lk_\Delta(v),\lk_{\partial \Delta}(v)\big) -(d+1) f_{-1}\big(\lk_\Delta(v),\lk_{\partial \Delta}(v)\big) \right\}\\
&= f_1(\Delta,\partial \Delta)-(d+1)f_0(\Delta,\partial \Delta) \qquad & {\mbox{(by \eqref{linksum})}}\\
&= g_2(\Delta,\partial \Delta),
\end{align*}
as desired.\end{proof}


In the rest of this section we treat the case of equality in Theorem \ref{main1} when $\Delta$ is an $\FF$-homology manifold with boundary. According to Theorem \ref{6.3}, this requires analyzing homology $(d-1)$-balls $B$ that satisfy ${d+2 \choose 2} \tilde \sigma_0\big(B,\partial B\big) = \frac 1 2 f_0\big(B,\partial B\big)$. To this end, we have: 

\begin{proposition} \label{missing-faces}
Let $B$ be a $(d-1)$-dimensional $\FF$-homology ball, where $d\geq 3$. Then
$$2{d+2 \choose 2} \tilde \sigma_0\big(B,\partial B\big) = f_0\big(B,\partial B\big)$$
if and only if $B$ can be written as $B=T\#\Sm_1\#\Sm_2\#\cdots\#\Sm_m$, where $T$ is a $(d-1)$-dimensional $\FF$-homology ball that has no interior vertices, $m=f_0(B,\partial B)$, and each $\Sm_i$ is the boundary complex of a $d$-dimensional simplex.
\end{proposition}

Proposition \ref{missing-faces} combined with Theorem \ref{6.3} and Corollary \ref{muinequality} implies the following criterion that, in particular, completes the proof of Theorem \ref{main1}.

\begin{corollary} \label{mu-equal-charact}
Let $\Delta$ be an $\FF$-homology $d$-manifold with boundary.
Then $$g_2(\Delta,\partial \Delta) = {d+2 \choose 2} \big(\mu_1(\Delta,\partial \Delta;\FF)-\mu_0(\Delta,\partial \Delta;\FF)\big)$$ if and only if $\Delta$ satisfies the following property (L):
\begin{enumerate}
\item[] the link of each interior vertex of $\Delta$ is a stacked sphere, and the link of each boundary vertex of $\Delta$ is obtained from an $\FF$-homology ball that has no interior vertices by forming connected sums with the boundary complexes of simplices.
\end{enumerate}
Moreover, if $g_2(\Delta,\partial \Delta) = {d+2 \choose 2} \big(\tilde{b}_1(\Delta,\partial \Delta;\FF)-\tilde{b}_0(\Delta,\partial \Delta;\FF)\big)$ then $\Delta$ satisfies property (L).
\end{corollary}

The proof of Proposition \ref{missing-faces} relies on the following three lemmas. Recall that for a simplicial complex $B$ and $F\subseteq V(B)$, $F$ is a \textbf{missing face of $B$} if every proper subset of $F$ is a face of $B$, but $F$ itself is \emph{not} a face of $B$. A missing face $F$ is a \textbf{missing $k$-face} if $|F|=k+1$. We denote by $m_k(B)$ the number of missing $k$-faces of $B$. Note that for a $(k+1)$-subset $F\subseteq V(B)$, $\tilde{b}_{k-1}(B_F)\in\{0,1\}$ and $F$ is a missing face of $B$ if and only if $\tilde{b}_{k-1}(B_F)=1$. In particular, 
\begin{equation} \label{m}
m_{d-1}(B)=\sum_{U\subseteq V(B), \ |U|=d} \tilde{b}_{d-2}\big(B_{U};\FF\big).
\end{equation} 

\begin{lemma} \label{balls-duality-cor}
Let $B$ be a $(d-1)$-dimensional $\FF$-homology ball, where $d\geq 3$. Then $B$ has at most $f_0(B,\partial B)$ missing $(d-1)$-faces. Furthermore, if $2{d+2 \choose 2} \tilde \sigma_0\big(B,\partial B\big) = f_0\big(B,\partial B\big)$, then $B$ has exactly $f_0(B,\partial B)$ missing $(d-1)$-faces.
\end{lemma}
\begin{proof} Assume that $V(B)=[n]$. According to Theorem \ref{4.9} and since $g_1(B,\partial B)=f_0(B,\partial B)$, 
\begin{align*}
f_0(B,\partial B)&\stackrel{(\dag)}{\geq} \beta^S_{n-d-1, n-d}\big(\FF[B,\partial B]\big) \qquad & \\
&= \sum_{W\subseteq [n], \ |W|=n-d} \tilde{b}_0\big(B_W,(\partial B)_W;\FF\big) \qquad &\mbox{(by Hochster's formula)}\\
&= \sum_{W\subseteq [n], \ |W|=n-d} \tilde{b}_{d-2}\big(B_{[n]\setminus W};\FF\big) \qquad &\mbox{(by eq.~\eqref{duality})}\\
&= m_{d-1}\big(B\big) \qquad &\mbox{(by eq.~\eqref{m})}.
\end{align*} 
Moreover, by Proposition \ref{6.1}, if $2{d+2 \choose 2} \tilde \sigma_0\big(B,\partial B\big) = f_0\big(B,\partial B\big)$, then $(\dag)$ holds as equality.
The assertion follows.
\end{proof}

For a finite set $A$, we denote by $\overline{A}:=\{C \, : \, C\subseteq A\}$ the simplex on $A$. We use the following well-known facts: (i) a homology sphere with a vertex removed is a homology ball, and (ii) if $M=K\#L$, then $M$ is a homology sphere if and only if $K$ and $L$ are homology spheres.

\begin{lemma} \label{boundary}
Let $T$ be a $(d-1)$-dimensional $\FF$-homology ball and let $\Sm$ be a $(d-1)$-dimensional $\FF$-homology sphere, where $d\geq 3$. If $\Delta=T\#_\phi\Sm$, then $\Delta$ is also an $\FF$-homology ball, $\partial\Delta=\partial T$, and $g_2(\Delta,\partial\Delta)=g_2(T,\partial T)+g_2(\Sm)$.
\end{lemma}
\begin{proof} 
Let $u\notin V(\Delta)$ be a new vertex. Define $K:=T\cup(u\ast\partial T)$ and $\Lambda:=K\#_{\phi}\Sm$. Since $T$ is an $\FF$-homology ball, $K$ is an $\FF$-homology sphere. Thus $\Lambda$ is the connected sum of homology spheres, and hence it is also a homology sphere. Finally, since $\Delta=T\#_\phi\Sm=\Lambda\setminus u$, it follows that $\Delta$ is a homology ball with $\partial\Delta=\lk_\Lambda(u)=\partial T$. As for the face numbers, assume $\Delta$ is formed by identifying facets $G\in T$ and $G'\in\Sm$. Then
$f_i(\Delta, \partial\Delta)=f_i(T,\partial T)+f_i(\Sm)-f_i(\overline{G})$ for all $-1\leq i\leq d-2$. Thus,
$g_2(\Delta,\partial\Delta)=g_2(T,\partial T)+g_2(\Sm)-g_2(\overline{G})=g_2(T,\partial T)+g_2(\Sm).$
\end{proof}

\begin{lemma} \label{connected_sum}
Let $d\geq 3$ and let $B$ be a $(d-1)$-dimensional $\FF$-homology ball with exactly $f_0(B,\partial B)$ missing $(d-1)$-faces.
Then $B=T\#\Sm_1\#\Sm_2\#\cdots\#\Sm_m$, where $T$ is a $(d-1)$-dimensional $\FF$-homology ball that has no interior vertices, $m=f_0(B,\partial B)$, and each $\Sm_i$ is the boundary complex of a $d$-dimensional simplex.
\end{lemma}
\begin{proof} The proof is by induction on $m:=f_0(B,\partial B)$ and $n:=f_0(B)$. If $m=0$, then $B$ has no interior vertices, so taking $T=B$ establishes the result in this case (for any value of $n$). If $m=1$ and $n=d+1$, then $B=\Sm\setminus G$ where $\Sm$ is the boundary of a $d$-simplex and $G$ is a facet of $\Sm$. In this case, $B=\overline{G}\#\Sm$, and we are done again.

Thus assume that either $m=1$ and $n>d+1$ (equivalently, $f_0(\partial B)>d$) or $m\geq 2$. Then $B$ has a missing $(d-1)$-face $G$ such that $\partial \overline{G}\neq \partial B$. Consider a new vertex $u\notin V(B)$, and let $\Lambda:=B\cup(u\ast \partial B)$. Then $\Lambda$ is a $(d-1)$-dimensional $\FF$-homology sphere, $G$ is a missing $(d-1)$-face of $\Lambda$, and $u\notin G$. By cutting along the boundary of $\overline{G}$ and filling in the two missing facets, $G_1$ and $G_2$, that result from $G$ (see Walkup \cite{Walkup-70} for more details on this operation), we conclude that $\Lambda$ can be expressed as $\Lambda_1\#_\phi \Lambda_2$, where $\Lambda_1$ and $\Lambda_2$ are $(d-1)$-dimensional $\FF$-homology spheres, $\phi$ identifies the vertices of $G_1\in\Lambda_1$ with the corresponding vertices of $G_2\in\Lambda_2$, and the indexing is chosen so that $u$ is a vertex of $\Lambda_1$. It follows that 
\begin{equation} \label{B}
B=\Lambda\setminus u = (\Lambda_1 \setminus u) \ \#_\phi \ \Lambda_2.
\end{equation}

Define $B_1:= \Lambda_1 \setminus u$ and $B_2:=\Lambda_2\setminus G_2$. Then $B_1$ and $B_2$ are $\FF$-homology $(d-1)$-balls. Furthermore, from the definition of $B_1$ and $B_2$ along with \eqref{B} and Lemma \ref{boundary} we infer that
\begin{itemize}
\item[(i)] $\partial{B_1}=\partial{B}$ and $\partial{B_2}=\partial\overline{G_2}$;
\item[(ii)] $f_0(B,\partial B)=f_0(B_1,\partial B_1)+f_0(B_2,\partial B_2)$;
\item[(iii)] $m_{d-1}(B)=m_{d-1}(B_1)+m_{d-1}(B_2)$;
\item[(iv)] $f_0(B_2,\partial B_2)=f_0(\Lambda_2)-d\geq 1$;
\item[(v)] $f_0(B_2)=f_0(B)-f_0(B_1)+d<f_0(B)$.
\end{itemize}
Here (v) follows from the observation that $B_1$ contains all the vertices of $\partial B_1=\partial B$ as well as of $G_1$, and from our choice of $G$ to satisfy $\partial\overline{G}\neq\partial{B}$. Since $m=f_0(B,\partial B)=m_{d-1}(B)$, we obtain from (ii), (iii), and Lemma \ref{balls-duality-cor} that for $i=1,2$, $B_i$ has exactly $f_0(B_i,\partial B_i)$ missing $(d-1)$-faces: $f_0(B_i,\partial B_i)=m_{d-1}(B_i)$; furthermore, by (iii) and (iv), $m_{d-1}(B_1)<m$, and by (iii) and (v), $m_{d-1}(B_2)\leq m$ while $f_0(B_2)<n$. Hence by the inductive hypothesis, 
$$B_1=T\#\Sm_1\#\cdots\#\Sm_k \quad \mbox{and} \quad B_2=D\#\Sm_{k+1}\#\cdots\#\Sm_m,
$$ where $T$ and $D$ are $\FF$-homology balls that have no interior vertices, $k=m_{d-1}(B_1)$, and
$\Sm_1,\ldots,\Sm_m$ are the boundary complexes of $d$-simplices. Also, since $\partial D=\partial B_2=\partial \overline{G_2}$ and since $D$ has no interior vertices, we conclude that $D=\overline{G_2}$. Eq.~\eqref{B} then yields that $B=T\#\Sm_1\#\cdots\#\Sm_m$.
\end{proof}

We are now ready to finish the proof of Proposition \ref{missing-faces}, and hence also of Corollary \ref{mu-equal-charact} and Theorem \ref{main1}.

\smallskip\noindent {\it Proof of Proposition \ref{missing-faces}: \ }
One direction of the proposition is immediate from Lemmas \ref{balls-duality-cor} and \ref{connected_sum}. For the other direction, assume that $B=T\#\Sm_1\#\cdots\#\Sm_m$, where $T$ is an $\FF$-homology $(d-1)$-ball without interior vertices, and $\Sm_1,\ldots, \Sm_m$ are the boundary complexes of $d$-simplices. We must show that $$2{d+2 \choose 2} \tilde \sigma_0\big(B,\partial B\big) = f_0\big(B,\partial B\big)=m.$$ We prove this by induction on $m$. If $m=0$, then $B$ has no interior vertices. Hence for every $W\subseteq V(B)$, the relative complex $(B_W, (\partial B)_W)$ has no $0$-faces. Consequently, $\tilde{b}_0(B_W, (\partial B)_W)=0$ for all $W$, and $\tilde \sigma_0\big(B,\partial B\big)=0$ by Definition \ref{def:sigma-and-mu}.

For the inductive step, it suffices to show that if $B$ is obtained from $B'$ by replacing a facet $G$ with $x\ast\partial\overline{G}$, where $x$ is a new vertex, then $$2{d+2 \choose 2}\tilde \sigma_0\big(B,\partial B\big)=2{d+2 \choose 2}\tilde \sigma_0\big(B',\partial B'\big)+ 1.$$ This is an easy consequence of Lemma \ref{1.3} and Lemma \ref{5.3}. Indeed, $\partial B=\partial B'$ while the $1$-skeleton of $B$ is obtained from the $1$-skeleton of $B'$ by adding a new vertex $x$ and connecting it to $d$ vertices of $B'$ (the vertices of $G$) which by themselves form a clique. Hence for $W\subseteq V(B')$,
\begin{align*}
\tilde{b}_0\big(B_{W\cup\{x\}}, (\partial B)_{W\cup\{x\}}\big) &=
\tilde{b}_0\big(B'_{W}, (\partial B')_{W}\big)+\left\{\begin{array}{ll}
1, & \mbox{if $W\cap G=\emptyset$}\\
0, &\mbox{if $W\cap G\neq \emptyset$}\end{array}\right. \\
&=\tilde{b}_0\big(B'_{W\cup\{x\}}, (\partial B')_{W\cup\{x\}}\big)+\left\{\begin{array}{ll}
1, & \mbox{if $W\cap G=\emptyset$}\\
0, &\mbox{if $W\cap G\neq \emptyset$}\end{array}\right. ,
\end{align*}
and we infer from Lemma \ref{1.3} that 
$$\tilde \sigma_0\big(B,\partial B\big)=\tilde \sigma_0\big(B',\partial B'\big)+\frac{1}{|V(B)|+1}\sum_{W\subseteq V(B')\setminus G}\frac{1}{{|V(B)|\choose |W|+1}}.$$ The result follows since  Lemma \ref{5.3} (applied with $n=|V(B)|$ and $r=1$) implies that the second summand equals $\frac 1 {(d+2) {d+1 \choose 1}}=\frac{1}{2{d+2\choose 2}}$.
\endproof

\section{Homology balls satisfying Theorem \ref{main1} with equality}

Theorem \ref{main1} provides necessary conditions on homology $d$-manifolds with boundary that 
satisfy $g_2(\Delta,\partial\Delta)={d+2 \choose 2}\big(\tilde{b}_1(\Delta,\partial\Delta)-\tilde{b}_0(\Delta,\partial\Delta)\big)$. The goal of this section is to establish a \textit{complete characterization} of homology balls with $g_2(\Delta,\partial\Delta)=0$. Specifically we prove

\begin{theorem} \label{criterion-ball}
Let $\Delta$ be a $d$-dimensional $\FF$-homology ball, where $d\geq 3$. Then $g_2(\Delta,\partial\Delta)=0$ if and only if $\Delta=\Gamma\#\Sm_1\#\cdots\#\Sm_m$, where $\Gamma$ is a $d$-dimensional $\FF$-homology ball that has no interior vertices and no interior edges, and each $\Sm_i$ is the boundary complex of a $(d+1)$-simplex. 
\end{theorem}

The proof is similar to the proof of Lemma \ref{connected_sum}, but requires one additional trick that we borrow from \cite{Swartz-08} and use in the following lemma.

\begin{lemma} \label{int-vertex}
Let $\Delta$ be a $d$-dimensional $\FF$-homology ball, where $d\geq 3$. If $g_2(\Delta,\partial\Delta)=0$ and if $\Delta$ has at least one interior vertex, then $\Delta$ has a missing $d$-face.
\end{lemma}
\begin{proof} Let $v_0$ be an interior vertex of $\Delta$. Since $g_2(\Delta,\partial\Delta)=0$, by Theorem \ref{6.3}, the link of $v_0$ is a stacked $(d-1)$-sphere. Hence either (i) $\lk_\Delta(v_0)$ is the boundary of a $d$-simplex $\overline{G}$, or (ii) $\lk_\Delta(v_0)$ has a missing $(d-1)$-face $H$. In case (i), $G$ is a required \emph{missing} $d$-face of $\Delta$: indeed, if $G$ were a face of $\Delta$, it would follow that $\{v_0\}\cup G$ is a missing $(d+1)$-face of $\Delta$, which is impossible since a homology $d$-ball cannot contain a $d$-sphere as a subcomplex. 

In case (ii), there are again two subcases to consider: either $H\in \Delta$ or $H\notin\Delta$. In the former case, $\{v_0\}\cup H$ is a missing $d$-face of $\Delta$, and we are done. 
To complete the proof, it suffices to show that the latter case is impossible. To achieve this, we follow the trick from \cite{Swartz-08}. Assume to the contrary that $H\notin\Delta$.
Since $\lk_\Delta(v_0)$ is a homology sphere, $\{H\}\cup \lk_\Delta(v_0)$ is the union of two homology spheres, $\Lambda_1$ and $\Lambda_2$, whose intersection is $\overline{H}$. Consider two new vertices $w_1,w_2\notin V(\Delta)$ and let $B:=(w_1\ast \Lambda_1)\cup (w_2\ast \Lambda_2)$. Then $B$ is a homology $d$-ball whose boundary coincides with $\lk_\Delta(v_0)=\partial(\st_\Delta(v_0))$. Define a new complex $\Delta'$ by deleting the star of $v_0$ from $\Delta$ and replacing it with $B$, that is, $\Delta':=(\Delta \setminus v_0)\cup B$. Then $\Delta'$ is also a homology $d$-ball; furthermore, $f_0(\Delta',\partial\Delta')=f_0(\Delta,\Delta')+1$, and $f_1(\Delta',\partial\Delta')=f_1(\Delta,\partial\Delta)+d$. Therefore, $g_2(\Delta',\partial\Delta')=g_2(\Delta,\partial\Delta)-1<0$, contradicting Theorem \ref{main1}. The result follows.
\end{proof}

\smallskip\noindent {\it Proof of Theorem \ref{criterion-ball}: \ } 
First, assume that $\Delta=\Gamma\#\Sm_1\#\cdots\#\Sm_m$ , where $\Gamma$ is a $d$-dimensional $\FF$-homology ball that has no interior vertices and no interior edges, and each $\Sm_i$ is the boundary complex of a $(d+1)$-simplex. Then $g_2(\Gamma,\partial\Gamma)=0$ (since $f_{i}(\Gamma,\partial \Gamma)=0$ for $i \leq 1$).
Also, $g_2(\Sm_i)=0$ for all $1\leq i\leq m$. Hence by Lemma \ref{boundary}, $g_2(\Delta,\partial\Delta)=0$, as required.

For the other direction, assume that $g_2(\Delta,\partial\Delta)=0$. We must show that $\Delta=\Gamma\#\Sm_1\#\cdots\#\Sm_m$ for $\Gamma$ and $\Sm_i$ as in the previous paragraph. We do this by induction on $f_0(\Delta,\partial\Delta)$. If $f_0(\Delta,\partial\Delta)=0$, then the condition $g_2(\Delta,\partial\Delta)=0$ implies that $f_1(\Delta,\partial\Delta)=0$, and so taking $\Gamma=\Delta$ yields the result.

Suppose $f_0(\Delta,\partial \Delta)>0$.
By Lemma \ref{int-vertex}, $\Delta$ has a missing $d$-face $G$.
Cutting and patching along this missing face $G$ as in the proof of Lemma \ref{connected_sum}, we find a homology ball $\Delta_1$ and a homology sphere $\Lambda_2$ such that $\Delta=\Delta_1\#_\phi \Lambda_2$, and $\phi$ identifies the facets $G_1\in \Delta_1$ and $G_2\in\Lambda_2$ resulting from $G$. Then by Lemma \ref{boundary},
\[
0=g_2(\Delta,\partial\Delta)=g_2(\Delta_1,\partial\Delta_1)+g_2(\Lambda_2).
\]
Hence, by the Lower Bound Theorem \cite{Kalai-87} and by Theorem \ref{main1}, $g_2(\Lambda_2)=g_2(\Delta_1,\partial\Delta_1)=0$ and $\Lambda_2$ is a stacked sphere. Furthermore, as in the proof of Lemma \ref{connected_sum}, $f_0(\Delta_1,\partial\Delta_1)<f_0(\Delta,\partial\Delta)$. Thus, by the inductive hypothesis, $\Delta_1$ is of the desired form $\Gamma\#\Sm_{1}\#\cdots\#\Sm_k$. In addition, since $\Lambda_2$ is a stacked sphere, it is also of the form $\Lambda_2=\Sm_{k+1}\#\cdots\#\Sm_m$, where each $\Sm_i$ is the boundary of a simplex and we can choose their ordering so that $G_2\in\Sm_{k+1}$. These decompositions of $\Delta_1$ and $\Lambda_2$ along with the fact that $\Delta=\Delta_1\#_\phi \Lambda_2$ imply the result.
\endproof

Let $\Delta$ be a $d$-dimensional homology ball with $d\geq 3$. Since by Proposition \ref{Dehn-Somm}, $h_i(\Delta, \partial \Delta)=h_{d+1-i}(\Delta)$, it follows that $g_2(\Delta,\partial\Delta)=0$ if and only if $h_{d-1}(\Delta)=h_d(\Delta)$. Consequently, Theorem \ref{criterion-ball} implies the following result on the $h$-vectors of balls:

\begin{proposition} The integer vector $h=(h_0, h_1, h_2, \ldots, h_{d-2}, m, m, 0)$ is the $h$-vector of a $d$-dimensional homology ball if and only if $h-(0,m,m,\ldots,m,0)$ is the $h$-vector of a $d$-dimensional homology ball. 
\end{proposition}

This result appears to be new. It is worth noting that if the Billera-Lee conjecture \cite{Billera-Lee-81} on the $h$-vectors of balls were true, it would imply the above statement. However, while Lee
and Schmidt \cite{Lee-Schmidt-12} confirmed this conjecture in dimensions three and four, Kolins
\cite{Kolins-11} disproved it in dimensions five and higher.


\section{Concluding remarks and open problems}
In this section we comment on some of our results as well as discuss several problems related to this paper that we left unsolved.

\subsection{The fundamental group and the $\mu$-numbers}
The lower bounds on the $g$-numbers in terms of the $\mu$-numbers established in Theorem \ref{5.5} are stronger than the corresponding lower bounds in terms of the Betti numbers given in Theorem~\ref{main2}. We now provide an application of  Theorem \ref{5.5} related to the fundamental group that does not follow from Theorem~\ref{main2}.

For a simplicial complex $\Delta$, we denote by $m(\Delta)$ the minimum number of generators of the fundamental group of $\Delta$. A long-standing conjecture of Kalai posits that an arbitrary connected $d$-dimensional homology manifold $\Delta$ without boundary (where $d\geq 3$) satisfies
$$g_2(\Delta) \geq  {d+2 \choose 2} m(\Delta).$$
It follows from the Hurewicz Theorem that $m(\Delta)$ is at least as large as the number of generators of $H_1(\Delta;\mathbb Z)$; consequently, $m(\Delta)\geq b_1(\Delta;\FF)$ for any field $\FF$. In fact, there exist acyclic complexes that are not simply-connected. Thus, the above conjecture of Kalai is stronger than the inequality of  \eqref{rig-ineq1}.
While some partial results on this conjecture were known, see \cite{Klee-09,Swartz-09}, even the question of whether $g_2(\Delta) \geq {d+2 \choose 2}$ holds for homology manifolds with a non-trivial fundamental group remained open. Here we use the $\mu$-numbers to settle this question in the affirmative.

Below we follow the same conventions as in Notation \ref{notation1}; all computations are done over a fixed field $\FF$.

\begin{lemma}
\label{9.1}
Let $\Delta$ be a simplicial complex with $V(\Delta)=V=\{v_1,\dots,v_n\}$.
\begin{enumerate}
\item[(i)] If $\Delta$ satisfies
\begin{align}
\label{simplyconnected}
\sum_{k=1}^n \big[\tilde b_0 \big(\lk_\Delta(v_k)_{V_{\leq k-1}}\big) -\tilde b_{-1}\big(\lk_\Delta(v_k)_{V_{\leq k-1}}\big)\big] = - b_0(\Delta),
\end{align}
then, for each $k=1,2,\dots,n$, all connected components of $\Delta_{V_{\leq k}}$ are simply-connected.
\item[(ii)] If $\Delta$ is an arbitrary connected but not simply-connected complex, then
$\mu_1(\Delta)-\mu_0(\Delta) \geq 0$.
\end{enumerate}
\end{lemma}

\begin{proof} We first note that part (ii) follows from part (i). Indeed, let $\Delta$ be a connected but not simply-connected complex.
By Lemma \ref{1.1}, the left-hand side of \eqref{simplyconnected} is at least as large  as $b_1(\Delta) -b_0(\Delta)\geq -b_0(\Delta)=-1$. Moreover, since $\Delta$ is not simply-connected, by part (i) of the statement, this inequality is strict, that is, the left-hand side of \eqref{simplyconnected} is $\geq 0$.  The result follows since $\mu_1(\Delta)-\mu_0(\Delta)$ is the average of the left-hand sides of \eqref{simplyconnected} over all possible orderings of $V$ (see the proof of Proposition \ref{1.2}).

We verify part (i). We may assume that each $v_j$ is indeed a vertex of $\Delta$.
Lemma \ref{1.1}(ii) and \eqref{simplyconnected} imply that $b_1(\Delta)=0$ and that the left-hand side of \eqref{simplyconnected} is equal to $b_1(\Delta)-b_0(\Delta)$. Thus, for $i=1$ (and all $k$), the inequality of \eqref{**3} must hold as equality. Hence we obtain from
the proof of Lemma \ref{1.1} that the following sequences are all exact:
\begin{align*}
0 &\longrightarrow H_1(\Delta_{V_{\leq k-1}}) \longrightarrow H_1(\Delta_{V_{\leq k}}) \longrightarrow \tilde H_0(\lk_\Delta(v_k)_{V_{\leq k-1}}) \\
&\longrightarrow H_0(\Delta_{V_{\leq k-1}}) \longrightarrow H_0(\Delta_{V_{\leq k}}) \longrightarrow \tilde H_{-1}(\lk_\Delta(v_k)_{V_{\leq k-1}})  \longrightarrow 0 \qquad \mbox{for all } k\leq n.
\end{align*}
Since $H_1(\Delta_{V_{\leq n}})=H_1(\Delta)$ vanishes, the exactness of these sequences  yields that $H_1(\Delta_{V_{\leq k}})=0$ for all $k$ and that the following property holds:
\begin{enumerate}
\item[$(\#)$] 
If $\lk_\Delta(v_k)_{V_{\leq k-1}} \ne \{\emptyset\}$, then $b_0(\Delta_{V_{\leq k-1}})-b_0(\Delta_{V_{\leq k}})=\tilde b_0(\lk_{\Delta}(v_k)_{V_{\leq k-1}})$.
\end{enumerate}

We now prove the statement by induction on $k$.
The complex $\Delta_{V_{\leq 1}}$ consists of a single vertex, and hence it is simply-connected.
Let $k \geq 2$ and suppose that all connected components of $\Delta_{V_{\leq k-1}}$ are simply-connected.
If $\lk_\Delta(v_k)_{V_{\leq k-1}} = \{ \emptyset\}$, then $\Delta_{V_{\leq k}}=\Delta_{V_{\leq k-1}} \uplus \overline{\{ v_k\}}$ and all connected components of $\Delta_{V_{\leq k}}$ are simply-connected. (Here $X \uplus Y$  denotes the disjoint union of $X$ and $Y$.)
Thus we may assume that $\lk_\Delta(v_k)_{V_{\leq k-1}} \ne \{ \emptyset\}$,
that $\Delta_{V_{\leq k-1}}$ has $s$ connected components which we denote by $\Gamma_1=\Delta_{W_1},\dots,\Gamma_s=\Delta_{W_s}$, and that $\lk_\Delta(v_k)_{V_{\leq k-1}}$ has $t$ connected components.
Property $(\#)$ implies that each $\Gamma_i$ contains at most one connected component of $\lk_\Delta(v_k)_{V_{\leq k-1}}$.  
Assume that $\Gamma_1,\dots,\Gamma_t$ contain connected components of $\lk_\Delta(v_k)_{V_{\leq k-1}}$ while $\Gamma_{t+1}\dots, \Gamma_s$ do not. 
Then
$$\Delta_{V_{\leq k}} = \big(\Delta_{W_1 \cup\{v_k\}} \cup \Delta_{W_2\cup\{v_k\}} \cup \cdots \cup \Delta_{W_t\cup\{v_k\}}\big)\uplus \Delta_{W_{t+1}} \uplus \cdots \uplus \Delta_{W_s}.$$
Note that the union $\Delta_{W_1 \cup\{v_k\}} \cup \Delta_{W_2\cup\{v_k\}} \cup \cdots \cup \Delta_{W_t\cup\{v_k\}}$ is the wedge sum since every two complexes $\Delta_{W_i\cup\{v_k\}}$ and $\Delta_{W_j\cup\{v_k\}}$ intersect exactly at $v_k$.
Since for $1\leq j\leq t$, the complexes in question satisfy (a) $\Gamma_j=\Delta_{W_j}$ is simply-connected, (b) $\lk_\Delta(v_k)_{W_j}$ is connected, and (c) $\st_{\Delta _{W_j \cup \{v_k\}}}(v_k)$ is contractable,
we obtain from the Seifert--van Kampen theorem that $\Delta_{W_j\cup\{v_k\}}$ is simply-connected (for all $j=1,2,\dots,t$).
Finally, since the wedge sum of simply-connected spaces is simply-connected, it follows that all connected components of $\Delta_{V_{\leq k}}$ are simply-connected.
\end{proof}

It was proved in \cite{Murai-15} (see the proof of Theorem 5.3 there) that if $\Delta$ is a $d$-dimensional normal pseudomanifold without boundary and $d \geq 3$, then
$$g_2(\Delta) \geq {d+2 \choose 2} (\mu_1(\Delta)-\mu_0(\Delta)+1).$$
This inequality along with Lemma \ref{9.1}(ii) yields the following strengthening of Swartz's result \cite[Theorem 4.3]{Swartz-09}.

\begin{theorem}
Let $\Delta$ be a connected $d$-dimensional normal pseudomanifold without boundary, where
$d \geq 3$. If $\Delta$ is not simply-connected, then
$g_2(\Delta) \geq {d+2 \choose 2}.$
\end{theorem}

\subsection{Characterizing the cases of equality}
One problem we have not been able to solve so far is that of characterizing homology manifolds (or even normal pseudomanifolds) that satisfy the inequality of Theorem \ref{main1} as equality. Theorem~\ref{main1} provides necessary conditions for the equality to hold. However, it is not hard to see that in dimension three these conditions are not sufficient. Indeed, if $\Delta$ is the boundary complex of a cyclic $4$-polytope with $n \geq 6$ vertices, then every vertex link of $\Delta$ is a stacked sphere, yet  $g_2(\Delta)>0$
(see \cite[Remark 8.5]{Kalai-87}). Therefore, if $\Gamma$ is defined as $\Delta$ with one facet removed, then (i) $\Gamma$ is a triangulated ball, (ii) each vertex link of $\Gamma$ is either a stacked sphere or a ball obtained from a simplex by forming connected sums with the boundary complexes of simplices, yet (iii) $g_2(\Gamma,\partial \Gamma)=g_2(\Delta)>0$. On the other hand, when $d \geq 4$, any $d$-dimensional homology manifold all of whose vertex links are stacked spheres satisfies the inequality of Theorem \ref{LBT-closed} as equality (see \cite[\S 9]{Kalai-87}). Thus we conjecture that for $d\geq 4$, the conditions on vertex links given in Theorem \ref{main1} do characterize homology manifolds $\Delta$ with $g_2(\Delta,\partial \Delta)={d +2 \choose 2} (\tilde b_1(\Delta,\partial \Delta)-\tilde b_0(\Delta,\partial \Delta))$. 
However, the downside of this conjectural criterion is that it is \textit{local} (i.e., stated in terms of the links). It would be more desirable to establish a \textit{global} criterion (say, in the spirit of Theorem \ref{criterion-ball}) characterizing homology manifolds that satisfy the inequality of Theorem \ref{main1} as equality.

It would also be very interesting to characterize the cases of equality in Theorems \ref{main2} and \ref{5.5}. For Theorem \ref{5.5}, it would be enough to characterize homology balls that satisfy the inequality of Proposition~\ref{5.2} as equality. Unfortunately, at the moment we do not even have a reasonable guess at such a characterization.
Also, in view of \cite[Conjecture 1]{Bagchi:mu-vector}, it seems plausible that, for a homology $(d-1)$-sphere $\Delta$, the equality holds in Proposition \ref{5.2} for some $j\leq\frac {d-1} 2$ if and only if $\Delta$ is $j$-stacked. The `if'-part of this conjecture was proved in \cite[Theorem 1.4]{Bagchi:mu-vector} while Proposition~\ref{5.2} and Theorem~\ref{GLBTsphere} verified the `only if'-part for all $j \leq \frac d 2 -1$ under an additional assumption that $\Delta$ has the WLP. However, a new approach seems to be required for the case of $j=\frac {d-1} 2$.
Of course, it would also be nice to remove the WLP assumption from the statement of Theorem~\ref{main2}.

\subsection{Additional remarks on $h_2$ and $g_2$}
It is also worth mentioning that our work on this paper was motivated by our desire to resolve Conjecture 5.7 of \cite{Novik-Swartz-09:DS}. Specifically, Theorem 5.1 of \cite{Novik-Swartz-09:DS} asserts that if $\Delta$ is a \emph{connected} $d$-dimensional homology manifold with nonempty orientable boundary and $d\geq 4$, then 
\begin{equation} \label{h_2-old}
h_2(\Delta)\geq f_0(\Delta,\partial\Delta)+\binom{d+1}{2}\tilde{b}_1(\partial\Delta)+(d+1)\tilde{b}_0(\partial\Delta),
\end{equation}
and Conjecture 5.7 of \cite{Novik-Swartz-09:DS} posits that equality occurs in eq.~\eqref{h_2-old} if and only if all vertex links of $\Delta$ are stacked spheres or stacked spheres with a vertex removed. Now, if $\Delta$ is connected, then $\tilde{b}_0(\Delta,\partial\Delta)=0$ while $\tilde{b}_1(\Delta,\partial\Delta)-\tilde{b}_0(\partial\Delta)\geq 0$ (as the connected homomorphism $\tilde{H}_1(\Delta,\partial\Delta) \to \tilde{H}_0(\partial\Delta)$ is onto in this case). Hence the inequality of Corollary \ref{h_2-vs-int-vert} is stronger than that of \eqref{h_2-old}. As a result, Corollary \ref{h_2-vs-int-vert} naturally leads to counterexamples to \cite[Conjecture 5.7]{Novik-Swartz-09:DS}. For instance, let $\Lambda$ be a $d$-dimensional stacked manifold with a non-vanishing $\tilde{b}_1(\Lambda)$, let $v$ be a vertex of $\Lambda$, and let $\Delta:=\Lambda\setminus v$. Then $\Delta$ is a $d$-dimensional manifold with boundary, and each vertex link of $\Delta$ is a stacked sphere or a stacked sphere with a vertex removed; also $\tilde{b}_1(\partial\Delta)= \tilde{b}_0(\partial\Delta)=0$. Therefore, if \cite[Conjecture 5.7]{Novik-Swartz-09:DS} were true, we would have $h_2(\Delta)=f_0(\Delta,\partial\Delta)$. However, according to Corollary~\ref{h_2-vs-int-vert} 
$$h_2(\Delta)\geq f_0(\Delta,\partial\Delta) + \binom{d+2}{2}\tilde{b}_1(\Delta,\partial\Delta)
= f_0(\Delta,\partial\Delta) + \binom{d+2}{2}\tilde{b}_1(\Lambda) > f_0(\Delta,\partial\Delta).
$$

In fact, the idea of deleting a vertex from a complex, along with Theorem \ref{main1}, leads to the following strengthening of Theorem~\ref{LBT-closed} with which we close this paper.

\begin{proposition} \label{LBT-closed-v2}
Let $\Gamma$ be a (not necessarily connected) $d$-dimensional normal pseudomanifold without boundary. If $d\geq 4$ and $v$ is a vertex of $\Gamma$, then 
$$g_2(\Gamma)\geq \binom{d+2}{2} \left(\tilde{b}_1(\Gamma; \field)-\tilde{b}_0(\Gamma; \field)\right) +
\binom{d+1}{2}\tilde{b}_1\big(\lk_\Gamma(v); \field\big).
$$
\end{proposition}
\begin{proof} Let $\Delta:=\Gamma\setminus v$. Then (i) $\Delta$ is a normal pseudomanifold with connected boundary: $\partial\Delta=\lk_\Gamma(v)$, and (ii) $\tilde{b}_i(\Gamma; \FF)=\tilde{b}_i(\Gamma, \st_\Gamma(v);\FF)=\tilde{b}_i(\Delta,\partial\Delta;\FF)$ (by contractibility of the star and by the excision theorem). Furthermore, a straightforward computation shows that
$$
g_2(\Gamma)=g_2(\Delta,\partial\Delta)+g_2(\partial\Delta). 
$$
The result follows by replacing $g_2(\partial\Delta)$ and $g_2(\Delta,\partial\Delta)$ with their lower bounds provided by Theorems \ref{LBT-closed} and \ref{main1}, respectively.
\end{proof}

We believe that Proposition \ref{LBT-closed-v2} can be further strengthened and look forward to seeing these results.

\bigskip

\noindent
\textbf{Acknowledgements.}
We are grateful to one of the referees for pointing to us the connection between the fundamental group and the $\mu$-numbers.

{\small
\bibliography{manifolds-biblio}
\bibliographystyle{plain}
}
\end{document}